\documentclass[11pt]{amsart}
\usepackage{amsmath,amssymb,amsthm,mathtools}
\usepackage{mathrsfs}
\usepackage{enumitem}
\usepackage{microtype}
\usepackage[margin=1in]{geometry}
\usepackage[hidelinks]{hyperref}
\newtheorem{theorem}{Theorem}[section]
\newtheorem{proposition}[theorem]{Proposition}
\newtheorem{lemma}[theorem]{Lemma}
\newtheorem{corollary}[theorem]{Corollary}
\theoremstyle{definition}
\newtheorem{definition}[theorem]{Definition}
\theoremstyle{remark}
\newtheorem{remark}[theorem]{Remark}
\numberwithin{equation}{section}
\newcommand{\R}{\mathbb R}
\newcommand{\B}{\mathbb B}
\newcommand{\Sph}{\mathbb S}
\newcommand{\C}{\mathbb C}
\newcommand{\II}{\mathrm{II}}
\newcommand{\ip}[2]{\langle #1,#2\rangle}
\newcommand{\proofstep}[2]{\par\smallskip\noindent Step~#1. \emph{#2}\ }
\DeclareMathOperator{\spanop}{span}
\author{Elham Matinpour}
\address{Department of Mathematics and Statistics, Auburn University, Auburn, AL 36849, USA}
\email{elm0135@auburn.edu}
\date{}
\title[Rigidity and Morse index of free boundary minimal $T$-surfaces]
{Rigidity and Morse index of free boundary minimal $T$-surfaces in the unit ball}
\begin{document}
\begin{abstract}
We study regular two-dimensional free boundary minimal $T$-surfaces in
the unit ball $\B^n$, with six faces, four triple junction arcs, and one
tetrahedral vertex. A conformal minimal immersion of the standard flat
$T$-cone, with the prescribed matching of junction traces, has the flat
cone as its image. We also obtain rigidity criteria that do not require
matching conformal parametrizations. In the disk-type class, two opposite planar faces suffice.
More generally, it is enough that a natural junction invariant vanish on
both junction edges of two opposite faces.

For variations preserving the junction identifications and keeping the
outer boundary on the sphere, we compute
\[
 \operatorname{ind}(TC)=n,
 \qquad \operatorname{nul}(TC)=3(n-2).
\]
The normal kernel consists exactly of the normal components of ambient rotations, and we compute the associated coupled harmonic Steklov spectrum. Every regular free boundary minimal $T$-surface with six faces has index at least $n$, with equality if and only if it is the flat $T$-cone; the same characterization holds for ambient variations.
Finally, the flat cone is locally rigid, modulo orthogonal
transformations, among sufficiently small compatible facewise $C^2$
free boundary minimal deformations.
\end{abstract}
\maketitle

\section{Introduction}
A free boundary minimal surface in the unit ball
$\B^n\subset\R^n$ has boundary on $\Sph^{n-1}$ and meets the sphere
orthogonally. Nitsche~\cite{nitsche1985stationary} proved rigidity for
free boundary minimal disks in $\B^3$, and Fraser and
Schoen~\cite{fraser2015uniqueness} extended the result to higher-dimensional
balls and space forms. Related boundary rigidity results for disk-type constant-mean-curvature and capillary surfaces with corners were obtained by Choe \cite{choe2002round}, using the Hopf differential and line-of-curvature boundary conditions. We study the tetrahedral singular analogue. Its
standard compact cone $TC$ has six flat sector faces, four triple junction
rays, one central $T$-point, and sector angle
\[
 \vartheta=\arccos(-1/3).
\]

Definition~\ref{def:regular-T} fixes the regular facewise setting used
throughout. The six abstract $C^{2,\alpha}$ faces are immersed with rank two
through edges and vertices, their traces agree on prescribed junctions,
the three outward conormals balance there, and the outer edges have
spherical free boundary. Images may have additional intersections, and
area and length are counted on the abstract faces. Related regularity and
uniformization results include
\cite{bernstein2021plateau,mese2006parameterized,mese2011local,wang2021uniformization,wang2022curvature,colombo2022polyhedral}.

\subsection{Conformal rigidity}
The first theorem keeps the fixed conformal structure of $TC$; in particular,
the junction traces match pointwise. Separately chosen conformal charts need
not induce the same junction parametrization
\cite{mese2006parameterized,mese2011local,wang2021uniformization}.

\begin{theorem}\label{thm:uniq.TC}
Let $n\geq3$ and $0<\alpha<1$. Let $TC$ be the standard compact
flat $T$-cone. Suppose that $u:TC\to\overline{\B^n}$ is continuous,
and that each face restriction is $C^{2,\alpha}$ on the closed
flat sector, conformal and minimal in its interior, and has rank two
everywhere, including at the vertices. Assume that the face maps
agree on the identified junctions, their outward unit junction
conormals sum to zero, and their outer arcs have spherical free
boundary. Then
\[
 u(TC)=A(TC)\qquad\text{for some }A\in\mathrm O(n).
\]
Each face covers the corresponding flat sector once.
\end{theorem}

The proof uses a scalar holomorphic quartic differential, with both orders
of each face used to group the junction terms into balanced triples on the
two rays of a reference sector. Together with the smooth disk theorem and
the conformal $Y$-cone theorem~\cite{matinpour2025cone}, this gives:

\begin{theorem}\label{thm:uniq.PlateauSurfaces}
Let $M$ be a closed planar unit disk, the standard compact $Y$-cone,
or the standard compact $T$-cone, and let $n\geq3$ and $0<\alpha<1$.
Suppose that the face restrictions of
$u:M\to\overline{\B^n}$ are $C^{2,\alpha}$, conformal and minimal,
have rank two through all edges and vertices, and have spherical free
boundary. In the singular cases assume pointwise agreement on the
model junctions and conormal balance there. Then $u(M)$ is congruent
to the corresponding flat model under an orthogonal transformation of
$\R^n$.
\end{theorem}

\subsection{Geometric rigidity and area}
The conformal matching assumption can be replaced by geometric
conditions. For a face $S_{ij}$ incident to junction $I_i$, let $T_i$
be a unit tangent and $\tau_{ij,i}$ its outward conormal, and define
\[
 c_{ij,i}=\ip{\II_{ij}(T_i,T_i)}{\II_{ij}(T_i,\tau_{ij,i})}.
\]
Two faces are called opposite when their index pairs are disjoint, for
example $S_{12}$ and $S_{34}$.
\begin{theorem}\label{thm:T-geometric-opposite}
Let $\Sigma$ be a regular disk-type free boundary minimal $T$-surface
in $\overline{\B^n}$. Any one of the following conditions implies that
$\Sigma$ is the standard flat $T$-cone up to an orthogonal
transformation, with every face covered once.
\begin{enumerate}[label=\textup{(\roman*)},leftmargin=2.4em]
 \item For each of two opposite faces $S_{ij}$, the two junction
       quantities $c_{ij,i}$ and $c_{ij,j}$ vanish on $I_i$ and $I_j$,
       respectively.
 \item Two opposite faces are planar.
 \item The outer boundary of each of two opposite faces contains a
       nonempty open subarc of a great circle of $\Sph^{n-1}$.
\end{enumerate}
No matching of the six separate conformal face parametrizations is
assumed.
\end{theorem}

The same class satisfies the Euler-characteristic identity of
Proposition~\ref{prop:T-area-general}. For six disk faces it becomes
\begin{equation}\label{eq:intro-area}
 2\operatorname{Area}(\Sigma)=\mathscr L
 =6\vartheta+\frac12\sum_{i<j}\int_{S_{ij}}|\II_{ij}|^2\,dA,
\end{equation}
where $\mathscr L$ is the total parametrized length of the six outer
arcs. This yields $\operatorname{Area}(\Sigma)\geq3\vartheta$, with equality
exactly for $TC$.

\subsection{Spectrum and Morse index}
Admissible variations preserve the junction identifications and keep the
outer edges on the sphere, while the junctions and the $T$-point may move.
Competitors need not remain minimal, conformal, balanced, or orthogonal to
the sphere. Let $\mathcal V_\Sigma$ be the facewise $H^1$ closure of their
normal variational fields. The second variation combines the standard
smooth free boundary terms with a coupled junction contribution; compare
\cite{fraser2016sharp,tran2020index,wang2022curvature}. For the flat model
we compute the complete coupled harmonic Steklov spectrum and the exact
index and normal nullity.

\begin{theorem}\label{thm:spectral-model-counts}
For the normal variation space $\mathcal V_{TC}$ described above,
\begin{equation}\label{eq:spectral-model-counts}
 \operatorname{ind}(TC)=n,\qquad
 \operatorname{nul}(TC)=3(n-2).
\end{equation}
Writing $\sigma_0<\sigma_1<\sigma_2<\cdots$ for the distinct
harmonic Steklov eigenvalues and $m_j$ for the multiplicity of
$\sigma_j$, the first three are
\[
 \sigma_0=0,\ m_0=n,\qquad
 \sigma_1=1,\ m_1=3(n-2),\qquad
 \sigma_2=\frac{\pi}{\vartheta},\ m_2=2(n-2).
\]
The zero Steklov fields are the normal projections of constant ambient
vectors. The entire normal kernel of $Q_{TC}$ consists of normal
components of ambient rotations.
\end{theorem}

Wang~\cite{wang2024networks} computes the index and nullity of stationary
triple junction networks in $\Sph^2$. Our Hessian acts on the full
free boundary cone and includes radial and higher-codimension modes; see
also~\cite{matinpour2025cone} for the $Y$-cone. Projected constant vectors
give $\operatorname{ind}(\Sigma)\ge n$, while the normal position field
$\Xi=X^\perp$ yields the equality case.

\begin{theorem}\label{thm:minimum-index-rigidity}
For admissible variations preserving the prescribed junction
identifications and keeping the outer edges on the sphere, a regular
free boundary minimal $T$-surface with six faces satisfies
\begin{equation}\label{eq:minimum-index-characterization}
 \operatorname{ind}(\Sigma)=n
 \quad\Longleftrightarrow\quad
 \Sigma=B(TC)\text{ for some }B\in\mathrm O(n).
\end{equation}
In the equality case every face covers its flat sector once.
Every nonflat surface in this class has index at least $n+1$.
\end{theorem}

No conformal matching or disk topology hypothesis is required in Theorem~\ref{thm:minimum-index-rigidity}. The same characterization holds for ambient variations; see Corollary~\ref{cor:minimum-ambient-index-rigidity}.

\subsection{Local rigidity and scope}
For local rigidity we restrict the area form to normal fields with zero trace on the spherical outer arcs while retaining the coupled junction conditions. We call this the Dirichlet form on the outer boundary, and its positivity also yields nonlinear local rigidity.

\begin{theorem}\label{thm:nonlinear-local-rigidity}
The standard free boundary $T$-cone is locally rigid modulo
orthogonal transformations in the facewise $C^2$ topology of
regular free boundary minimal $T$-surfaces with six faces, with the
junction identifications preserved.

More explicitly, let $p_{ij}$ be the reference maps in
\eqref{eq:sector-param}, for $i<j$. Suppose a surface in
Definition~\ref{def:regular-T} admits compatible face
parametrizations
$Y_{ij}=X_{ij}\circ\psi_{ij}:\overline{S_\vartheta}\to\overline{\B^n}$,
where each $\psi_{ij}$ is a $C^{2,\alpha}$ diffeomorphism respecting
the three boundary arcs, and the four prescribed junction
identifications are respected pointwise. If, for some $A\in\mathrm O(n)$,
\begin{equation}\label{eq:dilation-C2-small}
 \max_{i<j}\|A^{-1}Y_{ij}-p_{ij}\|_{C^2_*}<\frac1{10},
\end{equation}
then $\Sigma=B(TC)$ for some $B\in\mathrm O(n)$, with each face
covered once. Here
\[
 \|u\|_{C^2_*}
 =\max\left\{\sup|u|,\ \sup\|Du\|_{\mathrm{op}},\
                 \sup\left(\sum_{a,b=1}^2|\partial_{ab}u|^2\right)^{1/2}\right\}
\]
is computed in the ordinary Cartesian coordinates of the closed
sector. The maps $Y_{ij}$ are not assumed conformal, and the
central vertex is not assumed fixed at the origin.
\end{theorem}

This is a strong-topology result: it does not give rigidity from $C^1$
or varifold closeness, and the paper does not derive the
$C^{2,\alpha}$ rank two class from weak stationarity or classify all
regular free boundary minimal $T$-surfaces without one of the
conformal, geometric, index, or smallness hypotheses above.

\subsection{Organization}
Sections~\ref{Sec1-Preliminaries}--\ref{sec:geometric-rigidity} contain the
geometric and conformal results. Sections~\ref{sec:spectral-variations}--
\ref{sec:spectral-translations} develop the index form and spectrum, and
Sections~\ref{sec:dilation-rigidity}--\ref{sec:nonlinear-local-rigidity}
prove the remaining rigidity theorems.

\subsection*{Acknowledgment}
I would like to thank Professor Jaigyoung Choe for suggesting the problem.

\section{Preliminaries}\label{Sec1-Preliminaries}
Throughout, $\B^n$ is the open unit ball and $\Sph^{n-1}=\partial\B^n$.
We work with abstract faces carrying prescribed edge identifications;
distinct face images may intersect, and no ambient immersion, current
orientation, or local separation condition in $\R^n$ is assumed.

\subsection{The flat model and its sector coordinates}
Choose $v_1,\ldots,v_4\in\Sph^{n-1}$ with
\[
 \ip{v_i}{v_j}=-\frac13\quad(i\ne j),
 \qquad \sum_{i=1}^4v_i=0.
\]
Their span has dimension three. Put
\[
 \vartheta=\arccos(-1/3),\qquad
 e_{ij}=\frac{v_j-\cos\vartheta\,v_i}{\sin\vartheta}\quad(i\ne j).
\]
Let $\sigma_{ij}^0$ be the shorter great circle arc from $v_i$ to
$v_j$, and set
\begin{equation}\label{eq:model}
 \mathcal{NT}=\bigcup_{i<j}\sigma_{ij}^0,
 \qquad TC=\{r\omega:0\leq r\leq1,\ \omega\in\mathcal{NT}\}.
\end{equation}
Equivalently, if $E(P)$ is the straight one-skeleton of the
inscribed tetrahedron, then
$TC=\operatorname{cone}(E(P))\cap\overline{\B^n}$, where
$\operatorname{cone}$ uses all nonnegative radial multiples.
The six faces and four junctions are
\[
 F_{ij}=\{r\omega:0\leq r\leq1,\ \omega\in\sigma_{ij}^0\},
 \qquad \Gamma_i^0=[0,v_i].
\]
Here $F_{ij}=F_{ji}$ is one geometric face.

Use the sector
\[
 S_\vartheta=\{re^{i\theta}:0<r<1,\ 0<\theta<\vartheta\}.
\]
Its two closed radial edges are $\gamma_0=[0,1]$ and
$\gamma_\vartheta=e^{i\vartheta}[0,1]$, and its outer arc is
$\sigma_\vartheta=\{e^{i\theta}:0\leq\theta\leq\vartheta\}$.
For every ordered pair $i\ne j$, the linear isometry
\begin{equation}\label{eq:sector-param}
 p_{ij}(x+iy)=xv_i+ye_{ij}
\end{equation}
maps $\overline{S_\vartheta}$ onto $F_{ij}$. In particular,
\[
 p_{ij}(r)=rv_i,\qquad p_{ij}(re^{i\vartheta})=rv_j.
\]
The outward conormals along $\Gamma_i^0$ are $-e_{ij}$, and
\[
 \sum_{j\ne i}e_{ij}
 =\frac{-v_i-3\cos\vartheta\,v_i}{\sin\vartheta}=0.
\]
Each face is flat, and its outward conormal on the outer arc is
its position vector. It follows that the model is minimal and has spherical
free boundary.
With $R_\vartheta(z)=e^{i\vartheta}\overline z$, one has
\begin{equation}\label{eq:reversed-chart}
 p_{ji}=p_{ij}\circ R_\vartheta.
\end{equation}
The model metric is the Euclidean metric on each of these sectors,
with equal radial parameters identified at a junction. The compact
$Y$-model consists of three flat half-unit disks identified
pointwise along their diameters and meeting at $120^\circ$.

\subsection{Regular \texorpdfstring{$T$}{T}-surfaces}
Fix $0<\alpha<1$. A $C^{2,\alpha}$ map on a closed face with corners
means the restriction of such a map in ordinary extendible edge or corner
coordinates, and rank two at a corner refers to this differential. For a
flat model sector we use Cartesian coordinates; the extensions across
edges need not be minimal.

Take six compact connected $C^{2,\alpha}$ surfaces with corners
$S_{ij}=S_{ji}$, indexed by $i<j$. Each has one boundary component,
\[
 \partial S_{ij}=\gamma_{ij,i}\cup\gamma_{ij,j}\cup\sigma_{ij},
\]
consisting of two junction arcs and one outer arc. Their interiors
are disjoint, and consecutive arcs meet at three distinct abstract
vertices. Identify the junction arcs in triples by
$C^{2,\alpha}$ diffeomorphisms from four copies $I_i$ of $[0,1]$.
The two junction arcs of $S_{ij}$ meet at their $0$-endpoints. After
gluing, all these endpoints form the central vertex.
Their $1$-endpoints are the endpoints of $\sigma_{ij}$.
All boundary arcs are regular $C^{2,\alpha}$. No disk topology
assumption is imposed on $S_{ij}$.

\begin{definition}\label{def:regular-T}
A regular $T$-surface with six faces is a family of $C^{2,\alpha}$
immersions
\[
 X_{ij}:S_{ij}\longrightarrow\overline{\B^n},\qquad i<j,
\]
with rank two on each closed face, whose traces agree under the
prescribed junction identifications. Write these regular common
traces as $c_i:[0,1]\to\overline{\B^n}$, with
$c_i(0)=p$ and $c_i(1)=q_i$. The face interiors, $p$, and the open
junctions lie in $\B^n$, and the outer arcs lie in $\Sph^{n-1}$.
At $p$, the outgoing unit tangents $t_i=c_i'(0)/|c_i'(0)|$ satisfy
\begin{equation}\label{eq:T-tangent}
 \ip{t_i}{t_j}=-\frac13\quad(i\ne j),
\end{equation}
and the tangent sector of $S_{ij}$ is the sector of angle
$\vartheta$ between $t_i$ and $t_j$.

The surface is minimal if each face is minimal and, at every open
junction $I_i$, the outward unit conormals satisfy
\begin{equation}\label{eq:balance}
 \sum_{j\ne i}\tau_{ij,i}=0.
\end{equation}
It has spherical free boundary if its outward unit conormal
$\eta_{ij}$ on each outer arc satisfies
\begin{equation}\label{eq:freeboundary}
 \eta_{ij}=X_{ij}.
\end{equation}
\end{definition}

All metrics are induced, $\II_{ij}$ is the normal-vector-valued second
fundamental form, and tangents and conormals are identified with their
images under $dX_{ij}$. The notation $X_{ji}$ denotes the same map on
$S_{ji}=S_{ij}$, and extra coincidences of image points impose no further
identifications. Rank two and the stated regularity extend the edge
conormals and their relations continuously to the vertices.

The unit conormals in \eqref{eq:balance} meet at $120^\circ$. Because the
face traces agree, the induced metrics coincide on the common tangent line of
a junction. If $T$ is its unit tangent and
$\kappa_{ij,i}=\ip{D_TT}{\tau_{ij,i}}$, then
\[
 \sum_{j\ne i}\kappa_{ij,i}
 =\ip{D_TT}{\sum_{j\ne i}\tau_{ij,i}}=0.
\]

\begin{definition}\label{def:conformal}
A conformal minimal immersion of $TC$ is a continuous map with
face restrictions $u_{ij}=u\circ p_{ij}$ that are regular through
the closed sectors, conformal and minimal in their interiors,
and whose outward unit conormals satisfy \eqref{eq:balance}.
The common traces are
\begin{equation}\label{eq:traces}
 u_{ij}(r)=c_i(r),\qquad
 u_{ij}(re^{i\vartheta})=c_j(r),\qquad 0\leq r\leq1.
\end{equation}
Spherical free boundary means \eqref{eq:freeboundary} on the outer
arcs. In this definition regularity means $C^{2,\alpha}$ and rank
two on each closed sector, as in Theorem~\ref{thm:uniq.TC}.
\end{definition}

The two orders of the same face satisfy
\begin{equation}\label{eq:reversed-map}
 u_{ji}=u_{ij}\circ R_\vartheta.
\end{equation}
Both maps are harmonic and conformal; the coordinate change only reverses
the sector orientation and imposes no symmetry on the image.

\begin{lemma}\label{lem:containment}
A map satisfying the hypotheses of Theorem~\ref{thm:uniq.TC}
defines a regular free boundary minimal $T$-surface with six faces in
Definition~\ref{def:regular-T}. In particular, its vertex lies in
$\B^n$ and its tangent configuration there is tetrahedral.
\end{lemma}
\begin{proof}
\proofstep{1}{Identify the vertex differential.}
Set $a_i=c_i'(0)$. Conformality, continued to the vertex, implies
$|a_i|=|a_j|$ for every pair, since the two source rays are unit
vectors in one flat face. Write their common length as $\lambda>0$.
The same differential preserves their angle, so
\[
 \ip{a_i}{a_j}=-\lambda^2/3\quad(i\ne j).
\]
In particular,
\[
 \left|\sum_i a_i\right|^2
 =4\lambda^2+2\cdot6(-\lambda^2/3)=0.
\]
The unit vectors $t_i=a_i/\lambda$ give \eqref{eq:T-tangent}.
Their tangent sectors
are the images of the model sectors under conformal linear maps.

\proofstep{2}{Exclude the sphere from the face interiors and open junctions.}
For a face map,
\[
 \Delta |u_{ij}|^2
 =2\bigl(|(u_{ij})_x|^2+|(u_{ij})_y|^2\bigr)>0.
\]
The maximum principle excludes an interior point with image on the
sphere. Suppose an open junction point has common image $q$ with
$|q|=1$. The boundary point lemma applied on each incident sector
yields $\partial_\nu|u_{ij}|^2>0$ there. Conformality implies
$du_{ij}(\nu)=|c_i'|\tau_{ij,i}$ on that junction, so
$\ip{q}{\tau_{ij,i}}>0$ for all three faces, contradicting
\eqref{eq:balance}.

\proofstep{3}{Exclude the sphere from the common vertex.}
Suppose $p=u(0)$ satisfies $|p|=1$. Since each $c_i(r)$ lies in the
closed ball for $r\geq0$, one has $\ip{p}{a_i}\leq0$.
Their sum is zero, so each inner product is zero. On face $ij$,
the two source rays are linearly independent and their images under
the vertex differential are $a_i$ and $a_j$. The Cartesian
gradient of $f=|u_{ij}|^2$ vanishes at the vertex.
Since $\vartheta>\pi/2$, both positive Cartesian coordinate
rays lie in $\overline{S_\vartheta}$. Along them $f\leq f(0)$.
The second-order Taylor formula implies
$f_{xx}(0)\leq0$ and $f_{yy}(0)\leq0$. On the other hand,
continuity of the second derivatives and harmonicity give
\[
 \Delta f(0)=2\bigl(|(u_{ij})_x(0)|^2+|(u_{ij})_y(0)|^2\bigr)
 =4\lambda^2>0,
\]
a contradiction. The remaining properties are the stated
regularity, trace, balance, and free boundary hypotheses.
\end{proof}

\section{Tetrahedral stationary geodesic networks}
\label{sec:uniq-tsg-net}
The following result concerns embedded networks in the sphere, so
there are no unlisted crossings of their edges. This is a statement
about boundary networks and is not an embeddedness hypothesis on the
conformal immersion in Theorem~\ref{thm:uniq.TC}.

\begin{definition}
A stationary geodesic network in $\Sph^{n-1}$ is a finite embedded
graph whose edges are simple geodesic arcs with disjoint interiors.
No vertex lies in the interior of an edge. At each vertex exactly
three edges meet, and their outgoing unit tangents sum to zero.
\end{definition}

\begin{definition}
Such a network has tetrahedral combinatorics if it has four
distinct vertices and exactly one edge between each pair of
vertices. Equivalently, its abstract graph is the complete graph
on four vertices, denoted by $K_4$, and it has six edges.
\end{definition}

\begin{definition}
The regular tetrahedral network is $\mathcal{NT}$ from
\eqref{eq:model}, or an orthogonal image of it. It consists of
the six shorter great circle arcs between four unit vectors whose
pairwise inner products are $-1/3$.
\end{definition}

\begin{lemma}\label{lem: uniq. of RT-network}
Let $\mathcal N\subset\Sph^{n-1}$ be an embedded stationary
geodesic network with tetrahedral combinatorics. Then
$\mathcal N=A(\mathcal{NT})$ for some $A\in\mathrm O(n)$.
In particular, every edge is a shorter great circle arc of length
$\vartheta=\arccos(-1/3)$.
\end{lemma}
\begin{proof}
\proofstep{1}{Place the network in a great two-sphere.}
Fix a vertex $w_1$. The three outgoing unit tangents span a
two-plane, since they sum to zero and meet at $120^\circ$.
Let $W$ be its span together with $w_1$. Every edge starting at
$w_1$ lies in $W$, so all four vertices belong to $W$.
An edge with nonantipodal endpoints lies in their linear span,
and therefore in $W$. If an edge has antipodal endpoints, the
other two edges at either endpoint have nonantipodal endpoints.
Their tangents lie in $W$, and balance puts the third tangent in
$W$ as well. Every edge lies in the great two-sphere
$W\cap\Sph^{n-1}$.

\proofstep{2}{Determine the vertex inner products.}
The connected graph is embedded in a two-sphere, so its
complementary regions are disks. Euler's formula yields
$4-6+F=2$, hence $F=4$. The graph has no loops, multiple edges,
or bridges, so every boundary walk has at least three edges.
Each edge is counted twice in these walks, giving total count
twelve. Each boundary is a three-cycle. At each of its
vertices the complementary region occupies one of the three
sectors between the incident edges, so its angle is $2\pi/3$.

Orient the sphere and the boundary of one such disk so that the
disk lies to the left. Let $n_1,n_2,n_3\in W$ be the constant
unit left conormals to its three oriented geodesic edges. With
each edge parametrized by arclength, $n_k=x\times x'$, using the
cross product of the oriented three-space $W$. The geodesic equation
$x''=-x$ makes $n_k$ constant. The exterior turn at each vertex
is $\pi/3$, and therefore
\[
 \ip{n_1}{n_2}=\ip{n_2}{n_3}=\ip{n_3}{n_1}=\frac12.
\]
At a vertex $x$ between consecutive edges, the positive turn determines
$\ip{n_{k-1}\times n_k}{x}=\sin(\pi/3)>0$.
This fixes the sign of the radial cross product. With cyclic
vertex labels,
\[
 w_1=\frac{n_3\times n_1}{\sin(\pi/3)},\quad
 w_2=\frac{n_1\times n_2}{\sin(\pi/3)},\quad
 w_3=\frac{n_2\times n_3}{\sin(\pi/3)}.
\]
For example,
\[
 \ip{w_1}{w_2}
 =\frac{\ip{n_3}{n_1}\ip{n_1}{n_2}-\ip{n_3}{n_2}}
        {\sin^2(\pi/3)}
 =\frac{1/4-1/2}{3/4}=-\frac13.
\]
The other pairs give the same value. Applying this to the four
complementary triangles yields $\ip{w_i}{w_j}=-1/3$ for all
$i\ne j$. No choice of shorter arcs was used in this step.

\proofstep{3}{Identify the edges.}
The vertices now form a regular tetrahedron. At $w_i$, the three
unit tangents to the shorter arcs are
\[
 s_{ij}=\frac{w_j+w_i/3}{\sin\vartheta},\qquad j\ne i.
\]
They sum to zero, and this is their only linear relation.
A simple geodesic edge between these nonantipodal endpoints is
either the shorter arc or the complementary longer arc. Its
outgoing tangent is $\varepsilon_{ij}s_{ij}$ with
$\varepsilon_{ij}\in\{1,-1\}$. The same sign occurs at its other
endpoint. Balance implies that all three signs at a vertex agree.
Connectedness then makes all six signs equal. If they were all
negative, two edges starting at any vertex would both pass through
its antipode in their interiors, contradicting embeddedness.
All edges must be the shorter arcs. Equality of the Gram
matrices identifies the vertices, and then the arcs, by an
orthogonal transformation.
\end{proof}

\section{Planar rigidity from geodesic outer boundary}\label{sec:uniq. FBM TS with NT}

\begin{lemma}\label{lem:cauchy-plane}
Let $X:S\to\overline{\B^n}$ be a connected regular minimal face
with spherical free boundary on an open boundary arc $\sigma$.
If $X(\sigma)$ lies in a great circle, then the entire face lies
in the linear two-plane of that great circle. The same conclusion
holds if $\II$ vanishes on $\sigma$.
\end{lemma}
\begin{proof}
Suppose first that $X(\sigma)\subset P$ for a linear two-plane
$P$. For $a\in P^\perp$, the function $w=\ip{X}{a}$ is harmonic
with respect to the induced metric, and on $\sigma$ it satisfies
\[
 w=0,\qquad \partial_\eta w=\ip{\eta}{a}=\ip{X}{a}=0.
\]
To construct boundary isothermal coordinates near a point of
$\sigma$, solve a local Dirichlet problem for the induced metric
$g$, with nonnegative, nonzero boundary values supported on the
remaining artificial boundary. The solution $t$ is positive in
the interior and zero on the chosen subarc. Boundary regularity yields $t\in C^{2,\alpha}$ near the subarc,
and the Hopf lemma ensures $dt\ne0$ there. On a smaller simply
connected neighborhood choose a harmonic conjugate $s$, with
$ds=-*_gdt$ for a local orientation. This is possible because
$d(*_gdt)=0$. The differentials $ds,dt$ are orthogonal and have
the same nonzero length, so $(s,t)$ are isothermal coordinates
with the boundary given by $t=0$. The metric has the form
$g=\mu^2(ds^2+dt^2)$ with $\mu>0$. In these coordinates, $w$ is Euclidean harmonic,
and its vanishing conormal derivative is equivalent to
$\partial_t w=0$ on the boundary. Restrict to a coordinate half-disk.
Both its Dirichlet values and its normal derivative vanish on the diameter.
Odd reflection of the zero-Dirichlet function produces a harmonic
extension. Both its tangential and normal derivatives vanish on
the diameter. Its holomorphic complex derivative is
zero there, and the identity theorem makes it zero near the
subarc. Harmonic unique continuation on the connected face then shows
$w=0$ everywhere. Applying this to a basis of $P^\perp$ proves
$X(S)\subset P$.

For the last statement, parametrize the image of $\sigma$ by
arclength as $c(s)$. The vectors $c$ and $c'$ are an orthonormal
basis of the face tangent plane, by free boundary. If $\II=0$ on
the arc, then $c''$ is tangent. Since
\[
 \ip{c''}{c}=-1,\qquad \ip{c''}{c'}=0,
\]
one has $c''=-c$. The arc is contained in a great circle, so the
first part applies.
\end{proof}

\begin{lemma}\label{lem:outer-corners}
For a regular free boundary minimal $T$-surface, every junction
meets the sphere radially at its outer endpoint. The two outer
corner angles of each face are $\pi/2$.
\end{lemma}
\begin{proof}
At an outer endpoint $q_i$, let $T$ be the unit junction tangent.
The three unit conormals $\tau_1,\tau_2,\tau_3$ lie in $T^\perp$
and satisfy $\tau_1+\tau_2+\tau_3=0$, which gives
$\ip{\tau_a}{\tau_b}=-1/2$ for $a\ne b$, so they span a
two-plane in $T^\perp$. The three incident tangent planes are
$\spanop\{T,\tau_a\}$, and their common intersection is therefore
$\R T$. By the free boundary condition, $q_i$ belongs to each of
these tangent planes. Since $|q_i|=1$, we have $q_i=\pm T$. The outer boundary
tangent is perpendicular to $q_i$ and hence to the junction.
The inward junction ray is $-q_i$, since the other ray leaves
the ball. The regular corner lies in the supporting half-space
of the sphere, so its intrinsic angle is $\pi/2$.
\end{proof}

\begin{lemma}\label{lem:planar-T}
If all six faces of a regular free boundary minimal
$T$-surface with six faces are planar, then its image is $A(TC)$ for some
$A\in\mathrm O(n)$. Each face covers its flat sector once.
\end{lemma}
\begin{proof}
\proofstep{1}{Locate the face planes and the common vertex.}
Write the affine plane of face $S_{ij}$ as $b_{ij}+P_{ij}$.
At an outer boundary point $q$, free boundary means
$q=\eta_{ij}\in P_{ij}$, while $q\in b_{ij}+P_{ij}$.
It follows that $b_{ij}\in P_{ij}$, so the affine plane is the linear plane
$P_{ij}$ itself. At the common vertex $p$,
\[
 P_{ij}=\spanop\{t_i,t_j\}.
\]
The tetrahedral vectors $t_i$ span a three-plane, and any three
are linearly independent. For example,
$P_{12}\cap P_{13}=\R t_1$, whereas $t_1\notin P_{23}$.
It follows that $\bigcap_{i<j}P_{ij}=\{0\}$. Since $p$ lies in
every face plane, $p=0$.

\proofstep{2}{Identify the junction segments.}
The junction $c_i$ lies in
$\bigcap_{j\ne i}P_{ij}=\R t_i$. Write $c_i(s)=r_i(s)t_i$,
with $r_i(0)=0$. Its derivative never vanishes and initially
points along $t_i$, so $r_i$ is strictly increasing. Its other
endpoint is on the unit sphere, so $r_i(1)=1$. Each junction
is exactly $[0,t_i]$, covered once.

\proofstep{3}{Identify the outer arc and exclude multiple covering.}
The planar immersion pulls back an orientation from $P_{ij}$.
Use Gauss--Bonnet on the abstract face $S_{ij}$ with its
induced flat metric and inward geodesic curvature convention.
The straight junction edges have geodesic curvature zero.
On the outer edge, a unit-speed parametrization $c$ satisfies
$|c|=1$ and $\ip{c''}{-c}=1$. Since the inward conormal is
$-X_{ij}$, its geodesic curvature is one. The corner angles are
$\vartheta,\pi/2,\pi/2$, so its parametrized length $L_{ij}$
satisfies
\[
 L_{ij}+2\pi-\vartheta=2\pi\chi(S_{ij}).
\]
A compact connected surface with nonempty boundary has
$\chi(S_{ij})\leq1$. Since $L_{ij}>0$ and $0<\vartheta<\pi$,
the displayed identity implies $\chi(S_{ij})=1$, so the face is
a disk and $L_{ij}=\vartheta$.
The regular outer curve on the unit circle of $P_{ij}$ has an
angular lift with derivative of constant sign. Its length is
therefore its total angular variation. Since $L_{ij}=\vartheta<\pi$,
it follows that this curve is the shorter arc from $t_i$ to $t_j$,
covered once. The boundary map is a homeomorphism onto the boundary of
the closed flat sector $D_{ij}$ between these two unit rays.

The entire face, not only its boundary, equals this sector. In orthonormal coordinates $(a,b)$ on $P_{ij}$, put $t_i=(1,0)$
and $t_j=(\cos\vartheta,\sin\vartheta)$. The angular wedge is
specified by
\[
 b\geq0,\qquad a\sin\vartheta-b\cos\vartheta\geq0.
\]
The pullbacks of these two linear functions by $X_{ij}$ are
harmonic on the face and nonnegative on its boundary. Neither
vanishes identically, since that would put the rank two image in
a line. The strong minimum principle makes both positive in the
interior. Together with $|X_{ij}|<1$ in the interior,
this produces a local diffeomorphism
\[
 X_{ij}:\operatorname{int}S_{ij}\longrightarrow\operatorname{int}D_{ij}.
\]
It is proper because boundary points map to $\partial D_{ij}$
and $S_{ij}$ is compact. Its image is open and closed, hence is
all of $\operatorname{int}D_{ij}$. A proper local diffeomorphism
is a covering, and the target sector is simply connected, so this
covering has one sheet. The boundary was already covered once.
The six sectors with tetrahedral rays form $A(TC)$.
\end{proof}

\begin{theorem}\label{Thm:uniq. FBMT with NT}
Let $\Sigma$ be a regular free boundary minimal
$T$-surface with six faces in $\overline{\B^n}$. Suppose that every outer face arc
is a geodesic arc of the unit sphere. Then $\Sigma=A(TC)$ for some
$A\in\mathrm O(n)$, and every face is a disk covered once.
In particular, this conclusion holds when its spherical boundary
is the regular tetrahedral network, or an embedded stationary
geodesic network with tetrahedral combinatorics.
\end{theorem}
\begin{proof}
Every outer arc is contained in a great circle, so
Lemma~\ref{lem:cauchy-plane} makes each face planar and
Lemma~\ref{lem:planar-T} proves the conclusion; embeddedness of the
outer arcs is not needed for this.

If the boundary image is one of the stated finite geodesic
networks, each regular outer face arc has a nonempty subarc
avoiding the network vertices. That subarc lies in a single
geodesic edge. Applying Lemma~\ref{lem:cauchy-plane} to this subarc
again makes the entire face planar. This proves the two special
cases even when the boundary is specified only as an image set.
The network lemma independently identifies an embedded stationary
tetrahedral boundary with $A(\mathcal{NT})$.
\end{proof}

\section{Conformal rigidity of the free boundary \texorpdfstring{$T$}{T}-cone}
\label{sec:uniq. Tcone}
We use the ordered sector maps $u_{ij}$ from
\eqref{eq:sector-param} and \eqref{eq:traces}, and we extend the
Euclidean inner product complex bilinearly to $\C^n$; it is not the
Hermitian product. For $z=x+iy$, we use
$\partial_z=\tfrac12(\partial_x-i\partial_y)$ and
$\partial_{\bar z}=\tfrac12(\partial_x+i\partial_y)$.
The superscript $\perp$ denotes orthogonal projection to the
normal space of the relevant face at the specified parameter
point, extended complex linearly.

\begin{lemma}\label{lem:quartic}
For a regular conformal minimal sector map $v$, put
\[
 \varphi_v=\ip{v_{zz}}{v_{zz}},\qquad
 \mathcal H_v=v_{zz}^\perp.
\]
Then $\varphi_v$ is holomorphic in $S_\vartheta$ and continuous
on its closure, and
\begin{equation}\label{eq:quartic-square}
 \varphi_v=\ip{\mathcal H_v}{\mathcal H_v}.
\end{equation}
If $v$ has spherical free boundary on $\sigma_\vartheta$, then
\begin{equation}\label{eq:quartic-positive}
 z^4\varphi_v(z)=\frac14|v_{rr}^\perp|^2\geq0
 \qquad(z\in\sigma_\vartheta).
\end{equation}
\end{lemma}
\begin{proof}
Harmonicity and conformality give
\[
 v_{z\bar z}=0,\qquad \ip{v_z}{v_z}=0,
 \qquad \ip{v_{zz}}{v_z}=0.
\]
The harmonic map is smooth in the interior, so these identities
give holomorphicity of $v_{zz}$. The complexified tangent plane
is spanned by $v_z,v_{\bar z}$, and rank two implies
$\ip{v_z}{v_{\bar z}}=|v_x|^2/2>0$.
Its tangential projection is a multiple of the null vector $v_z$:
\[
 v_{zz}=v_{zz}^\perp+
 \frac{\ip{v_{zz}}{v_{\bar z}}}{\ip{v_z}{v_{\bar z}}}\,v_z.
\]
Taking the complex-bilinear square proves
\eqref{eq:quartic-square}, and hence the scalar $\varphi_v$ is
holomorphic; nothing is claimed about $v_{zz}^\perp$ as a map into
the fixed space $\C^n$. This scalar construction also appears in
Fraser and Schoen~\cite{fraser2015uniqueness}.

In polar coordinates,
\begin{equation}\label{eq:polar-normal}
 v_{zz}^\perp=\frac{e^{-2i\theta}}4
 \left(v_{rr}^\perp-\frac1{r^2}v_{\theta\theta}^\perp
                    -\frac{2i}{r}v_{r\theta}^\perp\right).
\end{equation}
On $r=1$, free boundary means $v_r=\rho v$ with $\rho=|v_r|>0$.
Differentiating gives $v_{r\theta}=\rho_\theta v+\rho v_\theta$.
Both terms are tangent to the face, since $v=\rho^{-1}v_r$,
so $v_{r\theta}^\perp=0$.
The polar harmonic equation
$v_{rr}+r^{-1}v_r+r^{-2}v_{\theta\theta}=0$ extends to the open
edges by continuity. Its normal part on $r=1$ yields
$v_{rr}^\perp+v_{\theta\theta}^\perp=0$ there, and hence
\[
 z^2\mathcal H_v=\frac12v_{rr}^\perp
 \qquad\text{on }\sigma_\vartheta.
\]
The right side is a real normal vector, and its square yields
\eqref{eq:quartic-positive}. All formulas extend to the two outer
corners by the assumed regularity. The coefficient is bounded
also at the central vertex, since the Cartesian second derivatives
are bounded on the closed sector.
\end{proof}

\begin{lemma}\label{lem:ordered-cancellation}
For a map satisfying Theorem~\ref{thm:uniq.TC}, define
\begin{equation}\label{eq:H-ordered}
 H(z)=z^4\sum_{i\ne j}\varphi_{ij}(z),\qquad
 \varphi_{ij}=\ip{(u_{ij})_{zz}}{(u_{ij})_{zz}}.
\end{equation}
The sum has twelve terms, one for each ordered pair. Then $H$ is
holomorphic in $S_\vartheta$, continuous on its closure, and real
on both radial edges. It is nonnegative on the outer arc and
satisfies $H(0)=0$.
\end{lemma}
\begin{proof}
\proofstep{1}{Derive the balances on both rays.}
Conformality yields
\[
 \ip{(u_{ij})_r}{(u_{ij})_\theta}=0,\qquad
 |(u_{ij})_\theta|=r|(u_{ij})_r|.
\]
For fixed $i$, the three maps on the first ray have common trace
$c_i(r)$. Their radial derivatives and second radial derivatives
are $c_i'(r)$ and $c_i''(r)$. Since the outward conormal of this
ray is $-r^{-1}\partial_\theta$, their image conormals are
\[
 \tau_{ij,i}=-\frac{(u_{ij})_\theta(r,0)}{r|c_i'(r)|}.
\]
Balance and its derivative in the common radial parameter give
\begin{equation}\label{eq:lower-balance}
 \sum_{j\ne i}(u_{ij})_\theta(r,0)=0,\qquad
 \sum_{j\ne i}(u_{ij})_{r\theta}(r,0)=0.
\end{equation}

On the second ray fix $j$ instead. The common trace is $c_j(r)$,
and the outward conormal of the parameter sector is
$+r^{-1}\partial_\theta$, so
\[
 \tau_{ij,j}=\frac{(u_{ij})_\theta(r,\vartheta)}{r|c_j'(r)|},
\]
and similarly
\begin{equation}\label{eq:upper-balance}
 \sum_{i\ne j}(u_{ij})_\theta(r,\vartheta)=0,\qquad
 \sum_{i\ne j}(u_{ij})_{r\theta}(r,\vartheta)=0.
\end{equation}
The same sign is used within each triple, the signs on the two rays
are opposite, and the ordered indexing makes all four junction
triples available on each ray.

\proofstep{2}{Compute the imaginary part.}
For any of these maps, the full polar formula is
\[
 u_{zz}=\frac{e^{-2i\theta}}4
 \left(u_{rr}-\frac{u_r}{r}-\frac{u_{\theta\theta}}{r^2}
       -\frac{2i}{r}u_{r\theta}+\frac{2i}{r^2}u_\theta\right).
\]
Substituting $u_{\theta\theta}=-r^2u_{rr}-ru_r$ yields
\begin{equation}\label{eq:polar-full}
 z^2(u_{ij})_{zz}
 =\frac12\left(r^2(u_{ij})_{rr}-iD_{ij}\right),\qquad
 D_{ij}=r(u_{ij})_{r\theta}-(u_{ij})_\theta.
\end{equation}
Both vectors on the right are real, and therefore
\begin{equation}\label{eq:imaginary}
 \operatorname{Im}(z^4\varphi_{ij})
 =-\frac{r^2}{2}\ip{(u_{ij})_{rr}}{D_{ij}}.
\end{equation}
On $\theta=0$, group the twelve terms by $i$. For each group,
$(u_{ij})_{rr}=c_i''$ and \eqref{eq:lower-balance} imply
$\sum_{j\ne i}D_{ij}=0$. It follows that
\[
 \operatorname{Im}H(r)
 =-\frac{r^2}{2}\sum_i
   \ip{c_i''(r)}{\sum_{j\ne i}D_{ij}(r,0)}=0.
\]
On $\theta=\vartheta$, group by $j$ and use
\eqref{eq:upper-balance} to obtain
\[
 \operatorname{Im}H(re^{i\vartheta})
 =-\frac{r^2}{2}\sum_j
   \ip{c_j''(r)}{\sum_{i\ne j}D_{ij}(r,\vartheta)}=0.
\]
Note that the phase $z^4$ is included in \eqref{eq:imaginary}, so no
reality of $z^4$ on the second ray is needed.

\proofstep{3}{Treat the remaining boundary and the vertex.}
Lemma~\ref{lem:quartic} provides holomorphicity, continuity, and
\begin{equation}\label{eq:boundary-sum}
 H(z)=\frac14\sum_{i\ne j}|(u_{ij})_{rr}^\perp|^2\geq0
 \qquad\text{on }\sigma_\vartheta.
\end{equation}
Bounded Cartesian second derivatives imply
$H(z)=O(|z|^4)$ at the central vertex, so $H(0)=0$.
Continuity extends the radial conclusions to the endpoints.
\end{proof}

\begin{proof}[Proof of Theorem~\ref{thm:uniq.TC}]
\proofstep{1}{The summed quartic vanishes.}
Use $H$ from Lemma~\ref{lem:ordered-cancellation}.
Its imaginary part is harmonic in the sector, continuous on the
closure, and zero on the entire boundary. The maximum principle
gives $\operatorname{Im}H=0$ in the sector. A real-valued
holomorphic function is constant. Since $H(0)=0$, we obtain
$H\equiv0$.

\proofstep{2}{Every face is planar.}
On the open outer arc, the sum of nonnegative real squares in
\eqref{eq:boundary-sum} vanishes term by term, so for each face
\[
 (u_{ij})_{rr}^\perp=0.
\]
Free boundary and normal harmonicity, as in
Lemma~\ref{lem:quartic}, also give
\[
 (u_{ij})_{r\theta}^\perp=0,\qquad
 (u_{ij})_{\theta\theta}^\perp=0.
\]
The two coordinate derivatives span the tangent plane there, so
$\II_{ij}=0$ on the open outer arc. Lemma~\ref{lem:cauchy-plane}
therefore makes each connected face planar. The inference uses
the real boundary norms in \eqref{eq:boundary-sum}. It does not
infer vanishing of a vector from its complex-bilinear square.

\proofstep{3}{Identify the full image.}
Lemma~\ref{lem:containment} supplies the regular $T$-surface
setting, and Lemma~\ref{lem:planar-T} gives
$u(TC)=A(TC)$, with each face covered once. No embeddedness of the
boundary network was needed before this conclusion.
\end{proof}

\begin{proof}[Proof of Theorem~\ref{thm:uniq.PlateauSurfaces}]
For a disk, $\Delta|u|^2=2(|u_x|^2+|u_y|^2)>0$ puts the
interior image in $\B^n$. Compactness of the closed source disk
and the spherical boundary values make the interior restriction
proper. The conclusion is therefore the theorem of Nitsche in
$\R^3$~\cite{nitsche1985stationary} and
Fraser and Schoen in arbitrary dimension~\cite{fraser2015uniqueness}.
For the $Y$-model, it is the conformal rigidity theorem
of~\cite{matinpour2025cone}, with its pointwise matching
of the three diameter traces. For the $T$-model, it is
Theorem~\ref{thm:uniq.TC}.
\end{proof}

\section{Area and curvature identities}\label{sec:area-curvature}
The geometric hypotheses in Definition~\ref{def:regular-T} already give
an area-curvature identity, without conformal matching of the face
parametrizations and without disk topology.

For a face $S_{ij}$ write
\[
 \mathscr A_{ij}=\operatorname{Area}(S_{ij}),\qquad
 \ell_{ij}=\operatorname{Length}(X_{ij}(\sigma_{ij})),\qquad
 \mathscr E_{ij}=\int_{S_{ij}}|\II_{ij}|^2\,dA,
\]
where area and length are counted on the abstract face. Put
\[
 \mathscr A=\sum_{i<j}\mathscr A_{ij},\qquad
 \mathscr L=\sum_{i<j}\ell_{ij},\qquad
 \mathscr E=\sum_{i<j}\mathscr E_{ij}.
\]

\begin{proposition}\label{prop:T-area-general}
Let $\Sigma$ be a regular free boundary minimal $T$-surface with six faces in
$\overline{\B^n}$. Then
\begin{equation}\label{eq:T-area-general}
 2\mathscr A=\mathscr L
 =6\vartheta-2\pi\sum_{i<j}\bigl(1-\chi(S_{ij})\bigr)
   +\frac12\mathscr E,
 \qquad \vartheta=\arccos(-1/3).
\end{equation}
If every face is a topological disk, then
\begin{equation}\label{eq:T-area-disk}
 2\mathscr A=\mathscr L=6\vartheta+\frac12\mathscr E,
 \qquad
 \mathscr A=3\vartheta+\frac14\mathscr E.
\end{equation}
For disk-type surfaces,
\[
 \mathscr A\geq3\vartheta,
\]
with equality if and only if $\Sigma$ is the standard flat $T$-cone up
to an orthogonal transformation, with every face covered once.
\end{proposition}
\begin{proof}
\proofstep{1}{Sum the Gauss--Bonnet formulas.}
Use the inward conormal convention for boundary geodesic curvature. On
an outer arc the inward conormal is $-X_{ij}$. If $Z$ is a unit tangent,
then
\[
 0=Z\ip{X_{ij}}{Z}=1+\ip{X_{ij}}{D_ZZ},
\]
so the outer geodesic curvature is
$\ip{D_ZZ}{-X_{ij}}=1$.
Along junction $I_k$ the inward conormal of face $S_{ij}$ is
$-\tau_{ij,k}$, hence its geodesic curvature is
$-\kappa_{ij,k}$.

The central corner angle of every face is $\vartheta$ by
Definition~\ref{def:regular-T}. Its other two corner angles are
$\pi/2$ by Lemma~\ref{lem:outer-corners}. The sum of the exterior
corner angles is
\[
 (\pi-\vartheta)+\frac\pi2+\frac\pi2=2\pi-\vartheta.
\]
Gauss--Bonnet on $S_{ij}$ yields
\[
 \int_{S_{ij}}K_{ij}\,dA+\ell_{ij}
 -\int_{I_i}\kappa_{ij,i}\,ds
 -\int_{I_j}\kappa_{ij,j}\,ds
 +(2\pi-\vartheta)=2\pi\chi(S_{ij}).
\]
For a minimal two-dimensional surface in Euclidean space, in arbitrary
codimension,
\[
 K_{ij}=-\frac12|\II_{ij}|^2.
\]
After summing the six formulas, the junction integrals cancel because at each $I_k$,
\[
 \sum_{j\ne k}\kappa_{kj,k}
 =\ip{D_TT}{\sum_{j\ne k}\tau_{kj,k}}=0.
\]
This gives
\begin{equation}\label{eq:T-length-curvature}
 \mathscr L=6\vartheta-2\pi\sum_{i<j}(1-\chi(S_{ij}))
   +\frac12\mathscr E.
\end{equation}

\proofstep{2}{Relate outer length and area.}
On a minimal two-dimensional face,
$\operatorname{div}_{S_{ij}}X_{ij}^{\top}=2$.
The divergence theorem yields
\[
 2\mathscr A_{ij}
 =\ell_{ij}
  +\int_{I_i}\ip{X_{ij}}{\tau_{ij,i}}\,ds
  +\int_{I_j}\ip{X_{ij}}{\tau_{ij,j}}\,ds.
\]
When the six identities are summed, the junction fluxes cancel pointwise by the common trace and the conormal balance, giving $2\mathscr A=\mathscr L$. Together with \eqref{eq:T-length-curvature} this proves
\eqref{eq:T-area-general}.

\proofstep{3}{The disk-type equality case.}
If every face is a disk, then $\chi(S_{ij})=1$, giving
\eqref{eq:T-area-disk}. Equality $\mathscr A=3\vartheta$ is equivalent
to $\mathscr E=0$. Since every $\II_{ij}$ is continuous, this means
that all six faces are planar. Lemma~\ref{lem:planar-T} then identifies
the image and the multiplicity. Conversely, the six sectors of the
standard cone have area $\vartheta/2$ each, so their total area is
$3\vartheta$.
\end{proof}

\section{Geometric rigidity without matching conformal parametrizations}
\label{sec:geometric-rigidity}
In this section every face is assumed to be a topological disk, and
the conformal coordinates are chosen separately on the six faces, so
no relation is imposed between the parameters induced on a common
junction.

For face $S_{ij}$ and junction $I_i$, choose either orientation of the
unit junction tangent $T_i$ and put
\begin{equation}\label{eq:T-c-invariant}
 c_{ij,i}
 =\ip{\II_{ij}(T_i,T_i)}{\II_{ij}(T_i,\tau_{ij,i})}.
\end{equation}
Changing the orientation of $T_i$ changes the sign of $c_{ij,i}$, so
the condition $c_{ij,i}\equiv0$ is orientation independent.

\subsection{Separate conformal coordinates and corner growth}
The next lemma collects the boundary regularity and corner estimates
needed below.

\begin{lemma}\label{lem:T-corner-growth}
Let $S$ be one face of a regular disk-type $T$-surface and
$X:S\to\R^n$ its $C^{2,\alpha}$ rank two immersion. Set
\[
 \mathbb H=\{z\in\mathbb C:\operatorname{Im}z>0\}.
\]
There is a conformal parametrization
\[
 \psi:\mathbb H\longrightarrow\operatorname{int}S,
 \qquad
 u=X\circ\psi:\mathbb H\longrightarrow\R^n,
\]
which extends homeomorphically to
$\overline{\mathbb H}\cup\{\infty\}$. Three marked points of
$\mathbb R\cup\{\infty\}$ map to the three corners of $S$, and the
three complementary boundary intervals map to the two junction edges
and the outer edge. On each open boundary interval, $\psi$ is
$C^{2,\alpha}$ with nonzero differential, and therefore $u$ is $C^{2,\alpha}$ and has rank two there.

Let $z_0$ be one of the three corner prevertices and let $\beta\pi$ be
the corresponding intrinsic corner angle. For a $T$-face,
$\beta\in\{\vartheta/\pi,1/2\}$, so $0<\beta<1$. If $z_0$ is finite,
then in a half-disk centered at $z_0$,
\begin{equation}\label{eq:T-corner-map-growth}
 |du(z)|\leq C|z-z_0|^{\beta-1}.
\end{equation}
Writing
\[
 \mathcal H=u_{zz}^{\perp},\qquad
 \varphi=\ip{\mathcal H}{\mathcal H},
\]
one has
\begin{equation}\label{eq:T-quartic-growth}
 |\mathcal H(z)|\leq C|z-z_0|^{2\beta-2},
 \qquad
 |\varphi(z)|\leq C|z-z_0|^{4\beta-4}.
\end{equation}
If a corner prevertex is at infinity, the same statements hold after
rewriting $u$ in the local half-plane coordinate $\zeta=-1/z$.
\end{lemma}
\begin{proof}
The induced metric is positive definite and $C^{1,\alpha}$ in the
extendible corner coordinates of Definition~\ref{def:regular-T}.
Extend the face slightly across its three boundary arcs and extend the
metric as a positive definite $C^{1,\alpha}$ metric to a simply
connected neighborhood $U$ of $S$. In dimension two, harmonic
coordinates and a harmonic conjugate give $C^{2,\alpha}$ isothermal
coordinates with nonsingular differential. We use here the standard
construction in \cite{deturck1981}.

Choose the isothermal charts consistently with an orientation of $U$.
Their transition maps are holomorphic, so they give $U$ a Riemann
surface structure. By uniformization, choose a biholomorphism
$F:U\to\mathcal U$, where
$\mathcal U$ is the disk or the plane. Since $S$ is compact and lies
in $U$, the set
\[
 \Omega=F(\operatorname{int}S)
\]
is a bounded Jordan domain. Its boundary consists of three
$C^{2,\alpha}$ arcs with the same corner angles as $S$. Let
$f:\mathbb H\to\Omega$ be a Riemann map. Carath\'eodory's theorem gives
a homeomorphic extension to the compactified boundary. A real
M\"obius transformation of $\mathbb H$ places the three prevertices at
three prescribed points of $\mathbb R\cup\{\infty\}$. Set
$\psi=F^{-1}\circ f$.

On each open boundary interval, the smooth-boundary regularity theorem
for conformal maps, applied also to the inverse map, makes $f$ of class
$C^{2,\alpha}$ with nonzero boundary derivative. At a prevertex with
interior angle $\beta\pi$, the standard corner asymptotics give
\[
 |f'(z)|\leq C|z-z_0|^{\beta-1}.
\]
We use the classical boundary and corner regularity theory in
\cite{pommerenke1992} for these two facts.
The maps $F^{-1}$ and $X$ have bounded first derivatives near the
corresponding corner. This proves \eqref{eq:T-corner-map-growth}.

The second fundamental form of $X$ is bounded in the geometric corner
coordinates. Since $u$ is conformal,
\[
 u_{zz}^{\perp}=\II(u_z,u_z).
\]
It follows that
\[
 |u_{zz}^{\perp}|\leq C|u_z|^2
 \leq C|z-z_0|^{2\beta-2},
\]
and the complex-bilinear square gives
\eqref{eq:T-quartic-growth}. If a prevertex is at infinity, use
$\zeta=-1/z$ and repeat the same local argument in that coordinate.
\end{proof}

For our faces the central angle is $\vartheta$ and the two outer angles
are $\pi/2$. The corresponding exponents are
\[
 \beta_0=\frac\vartheta\pi,\qquad
 \beta_1=\beta_2=\frac12.
\]
Since $\pi/2<\vartheta<2\pi/3$, one has
\begin{equation}\label{eq:T-pole-orders}
 4\beta_0-4\in(-2,-4/3),
 \qquad 4\beta_1-4=4\beta_2-4=-2.
\end{equation}
The first bound will exclude a pole of order two at the central
prevertex, while the two right angle corners allow poles of order at
most two.

\subsection{A facewise quartic reflection criterion}
\begin{lemma}\label{lem:T-two-junction-planarity}
Let $S_{ij}$ be a disk face of a regular free boundary minimal
$T$-surface, and use the quantities $c_{ij,i}$ and $c_{ij,j}$ from
\eqref{eq:T-c-invariant}. If
\begin{equation}\label{eq:T-two-c-zero}
 c_{ij,i}\equiv0\quad\text{on }I_i,
 \qquad
 c_{ij,j}\equiv0\quad\text{on }I_j,
\end{equation}
then the whole face $X_{ij}(S_{ij})$ lies in a linear two-plane through
the origin.
\end{lemma}
\begin{proof}
Choose the separate half-plane parametrization of
Lemma~\ref{lem:T-corner-growth}. On an open boundary interval write
$z=x+iy$, with $y=0$, and let $T$ and $\nu$ be the unit tangent and
outward unit conormal of the face. If
$\lambda=|u_x|=|u_y|$, then
$T=u_x/\lambda$ after choosing the boundary orientation and
$\nu=\pm u_y/\lambda$. Minimality gives
$\II(\nu,\nu)=-\II(T,T)$. Taking the normal part of $u_{zz}=\frac14(u_{xx}-u_{yy}-2iu_{xy})$ gives
\begin{equation}\label{eq:T-quartic-boundary-trace}
 u_{zz}^{\perp}
 =\frac{\lambda^2}{2}
   \bigl(\II(T,T)\mp i\,\II(T,\nu)\bigr).
\end{equation}
Taking the complex-bilinear square,
\begin{equation}\label{eq:T-quartic-imaginary-c}
 \operatorname{Im}\varphi
 =\mp\frac{\lambda^4}{2}
       \ip{\II(T,T)}{\II(T,\nu)}.
\end{equation}
On either junction interval, \eqref{eq:T-two-c-zero} makes
$\varphi$ real.

On the outer interval, free boundary means $\nu=X$, so
$D_T\nu=D_TX=T$ is tangent to the face. Taking its normal component in
the Gauss formula gives
\[
 \II(T,\nu)=0.
\]
The function $\varphi$ is real there as well, and
\begin{equation}\label{eq:T-outer-quartic-positive}
 \varphi=\frac{\lambda^4}{4}|\II(T,T)|^2\geq0
 \qquad\text{on the open outer interval}.
\end{equation}

As in Lemma~\ref{lem:quartic}, the scalar identity
$\varphi=\ip{u_{zz}}{u_{zz}}$ shows that the quartic differential
\[
 q=\varphi(z)\,dz^4
\]
is holomorphic in the half-plane. The notation $\varphi$ below always
means the coefficient of this same differential in the local source
coordinate being used. Since its coefficient is real on
each of the three open boundary intervals, Schwarz reflection extends
$q$ holomorphically across those intervals. The only possible
singularities on the doubled surface are the three prevertices. The
double of the half-plane is the Riemann sphere. The growth estimates
of Lemma~\ref{lem:T-corner-growth} show that each reflected coefficient
has finite-order growth at a prevertex, so the isolated singularity is
meromorphic there. Together with \eqref{eq:T-pole-orders}, they give a
pole of order at most one at the central prevertex and poles of order
at most two at the two outer prevertices. Indeed, a pole of order $m$
is excluded whenever the coefficient is $o(|z-z_0|^{-m})$. At the
central point the exponent is strictly greater than $-2$. The total
pole order is at most
\[
 1+2+2=5.
\]
A nonzero meromorphic quartic differential on the Riemann sphere has
divisor degree $-8$. Its total pole order is at least eight.
The preceding bound is impossible unless $q\equiv0$.

Equation~\eqref{eq:T-outer-quartic-positive} now implies
$\II(T,T)=0$ on the open outer edge. The mixed term is already zero by
free boundary, and minimality supplies the remaining diagonal component.
The full second fundamental form vanishes on that open boundary
arc. Lemma~\ref{lem:cauchy-plane} makes the connected face planar.
Its
plane is linear because of the free boundary condition, as in the
proof of Lemma~\ref{lem:planar-T}.
\end{proof}

\subsection{Propagation across a tetrahedral junction}
\begin{lemma}\label{lem:T-planar-propagation}
Fix one junction $I_i$. If one of its three incident faces is planar,
then for each of the other two incident faces the quantity
$c_{ij,i}$ in \eqref{eq:T-c-invariant} vanishes identically on $I_i$.
\end{lemma}
\begin{proof}
Choose an orientation of the unit junction tangent $T$ and suppress the
junction index. For each incident face put
\[
 K=D_TT,\qquad
 \kappa_j=\ip{K}{\tau_j},\qquad
 A_j=\II_j(T,T),\qquad
 B_j=\II_j(T,\tau_j).
\]
The Gauss formula gives
\[
 K=\nabla^{S_j}_T T+A_j=\kappa_j\tau_j+A_j.
\]
Since $T$ and $\tau_j$ are an orthonormal tangent frame along the junction,
\[
 \nabla^{S_j}_T\tau_j=-\kappa_jT,
 \qquad
 D_T\tau_j=-\kappa_jT+B_j.
\]
Thus
\begin{equation}\label{eq:T-AB-junction}
 A_j=K-\kappa_j\tau_j,
 \qquad
 B_j=D_T\tau_j+\kappa_jT.
\end{equation}
Since $K\perp T$, $\tau_j\perp T$, and
$\ip{\tau_j}{D_T\tau_j}=0$, one obtains the geometric identity
\begin{equation}\label{eq:T-c-KDtau}
 c_j=\ip{A_j}{B_j}=\ip{K}{D_T\tau_j}.
\end{equation}

Suppose face $1$ is planar. Then $A_1=B_1=0$, so
\[
 K=\kappa_1\tau_1,
 \qquad D_T\tau_1=-\kappa_1T.
\]
For either other face $j$, differentiate the constant-angle relation
$\ip{\tau_1}{\tau_j}=-1/2$:
\[
 0=\ip{D_T\tau_1}{\tau_j}+\ip{\tau_1}{D_T\tau_j}
   =\ip{\tau_1}{D_T\tau_j}.
\]
Equation~\eqref{eq:T-c-KDtau} then shows $c_j=0$ along the whole
junction.
\end{proof}

\begin{proof}[Proof of Theorem~\ref{thm:T-geometric-opposite}]
For (i), Lemma~\ref{lem:T-two-junction-planarity} makes the two
opposite faces planar. It remains to prove (ii).

Suppose, after relabeling, that $S_{12}$ and $S_{34}$ are planar.
Every one of the four junctions is incident to one of these two faces.
Lemma~\ref{lem:T-planar-propagation} then makes the quantity
$c$ vanish on every junction edge of each of the remaining four
faces. Each such face has both of its junction quantities equal to
zero, so Lemma~\ref{lem:T-two-junction-planarity} makes it planar.
All six faces are planar, and Lemma~\ref{lem:planar-T} identifies the
standard cone and the one-sheeted conclusion.

For (iii), the proof of Lemma~\ref{lem:cauchy-plane} is local on the
open spherical boundary arc. If a nonempty open subarc lies in a great
circle with plane $P$, then for every $a\in P^\perp$ the harmonic
function $\ip{X}{a}$ and its normal derivative vanish on that subarc.
The same
reflection and unique-continuation argument puts the entire connected
face in $P$. The two opposite faces are planar, so (ii) applies.
\end{proof}

\begin{remark}\label{rem:T-geometric-scope}
The disk assumption in Theorem~\ref{thm:T-geometric-opposite} is used
only to choose a global separate conformal half-plane coordinate on a
face and double its scalar quartic differential. It does not identify
junction parameters between different faces. The argument allows
immersed faces and additional intersections of their images. It also
does not assert that one planar face alone determines the full
$T$-surface.
\end{remark}

\section{Normal variations and the area Hessian}\label{sec:spectral-variations}
Let $n\geq3$, and let $\Sigma$ be a regular free boundary minimal
$T$-surface with six faces as in Definition~\ref{def:regular-T}. All
integrals below are taken on the abstract faces and junctions, so area
is counted with multiplicity, and only the given junction
identifications are used; extra intersections of the images impose no
additional constraints.

\subsection{The variation space}
An admissible variation is a family $X_{ij,t}$ of immersions of the
same abstract faces, defined for both signs of $t$ near zero, such that
\[
 X_{ij,0}=X_{ij},\qquad X_{ij,t}(S_{ij})\subset\overline{\B^n},
 \qquad X_{ij,t}(\sigma_{ij})\subset\Sph^{n-1}.
\]
The maps agree on all prescribed junctions, including the common
central vertex. The family is twice continuously differentiable in
$t$ with values in $C^1$, and its variational vector field is facewise
$C^{1,\alpha}$. Competitors need not be minimal, conformal, balanced,
or orthogonal to the sphere. The junctions and the central vertex may
move. There are no volume constraints or junction energies. Since the
abstract faces are compact, no separate support condition is imposed.

Write $U_{ij}=\partial_tX_{ij,t}|_{t=0}$ for the variational vector field and
$\Phi_{ij}=P_{ij}^{\perp}U_{ij}$ for its normal component, where
$P_{ij}^{\perp}$ denotes normal projection on the face at the specified
abstract point. Along junction $I_i$ the variational vector fields have
one common value $W_i$, and hence
\begin{equation}\label{eq:spectral-common-variation}
 \Phi_{ij}=P_{ij}^{\perp}W_i\quad\hbox{on }I_i\quad(j\ne i),
 \qquad \ip{U_{ij}}{X_{ij}}=0\quad\hbox{on }\sigma_{ij}.
\end{equation}
For a regular variational vector field the four fields $W_i$ agree at
the central vertex. These conditions are necessary, but we do not use
them to characterize all regular normal fields; a regular normal field
is, by definition, the normal component of a regular variational
vector field on the faces, and the form domain is defined by closure
below. For the flat cone, Proposition~\ref{prop:spectral-form-domain}
identifies that closure without imposing a pointwise condition on a
Sobolev representative at the center.

\begin{lemma}\label{lem:spectral-full-field-lift}
Every facewise $C^{1,\alpha}$ vector field $U$ with common junction
values, including at the center, and with $\ip{X}{U}=0$ on the
outer edges is the variational vector field of an admissible variation.
In particular, the admissible normal variational fields form a vector
space.
\end{lemma}
\begin{proof}
On each face put $\rho=1-|X|^2$. Along an outer edge the outward
conormal is $\eta=X$, so the derivative of $\rho$ in the inward
conormal direction is $2$. The right angle geometry of
Lemma~\ref{lem:outer-corners} gives the same transversality at the two
endpoints. In a sufficiently small outer collar with inward
coordinate $s\geq0$, one has $\rho\geq cs$. Since $\ip{X}{U}$
vanishes at $s=0$ and is $C^1$, $|\ip{X}{U}|\leq Cs$ there. On the
compact remainder $\rho$ has a positive minimum. After taking the
largest constant over the six faces,
$|\ip{X}{U}|\leq C_0\rho$ everywhere.

For each face define
\[
 X_\varepsilon=\frac{X+\varepsilon U}
                    {\sqrt{1+\varepsilon^2|U|^2}}.
\]
Its variational vector field is $U$. Along a junction, both $X$ and
$U$ have common values, so the scalar denominator is common as well and the maps retain the prescribed junction identifications. On an
outer edge, $|X|=1$ and $\ip{X}{U}=0$, so $|X_\varepsilon|=1$. In addition,
\[
 1-|X_\varepsilon|^2
 =\frac{\rho-2\varepsilon\ip{X}{U}}
        {1+\varepsilon^2|U|^2}
 \geq\frac{(1-2C_0|\varepsilon|)\rho}
              {1+\varepsilon^2|U|^2}\geq0
\]
when $2C_0|\varepsilon|<1$. Rank two persists for small
$|\varepsilon|$ by uniform $C^1$ convergence on the compact faces.
The denominator is positive and depends smoothly on $\varepsilon$
with values in $C^1$, so the family has the required differentiability.
Linear combinations of such vector fields satisfy the same linear
junction and boundary conditions. Applying the construction again
shows that their normal components form a vector space.
\end{proof}

Let $\mathcal V_\Sigma$ be the closure of these normal variational fields in
$\bigoplus_{i<j}H^1(N\Sigma_{ij})$. The $H^1$ norm uses the induced
metric and the normal connection. The area Hessian is a continuous
symmetric bilinear form $Q_\Sigma$ on this space, as proved below.
Its index $\operatorname{ind}(\Sigma)$ is the maximal dimension of a
subspace on which it is negative definite. Its normal kernel and nullity are
\begin{equation}\label{eq:spectral-kernel-definition}
 \ker Q_\Sigma=\{\Phi\in\mathcal V_\Sigma:
 Q_\Sigma(\Phi,\Psi)=0\text{ for every }\Psi\in\mathcal V_\Sigma\},
 \qquad \operatorname{nul}(\Sigma)=\dim\ker Q_\Sigma.
\end{equation}
Tangential reparametrizations are not counted, and a zero value of
$Q_\Sigma(\Phi,\Phi)$ does not by itself put $\Phi$ in the kernel.
The index agrees with the index tested on regular admissible normal
fields: a basis of a finite-dimensional negative subspace can be
approximated in $H^1$ by such fields, continuity of $Q_\Sigma$
preserves negative definiteness of the finite Gram matrix for
sufficiently close approximations, and the reverse inequality follows
from inclusion.

\subsection{The second variation}
For a local tangent orthonormal frame $(e_1,e_2)$ on a face, set
\[
 \mathcal A_{ij}\Phi=\sum_{a,b=1}^2
 \ip{\II_{ij}(e_a,e_b)}{\Phi}\,\II_{ij}(e_a,e_b),\qquad
 L_{ij}=\Delta_{ij}^{\perp}+\mathcal A_{ij}.
\]
Here $\Delta^\perp$ is the connection Laplacian with the sign
$\int|\nabla^\perp\Phi|^2=-\int\ip{\Delta^\perp\Phi}{\Phi}$ for
compactly supported fields. The endomorphism $\mathcal A_{ij}$ is
independent of the chosen frame. In codimension one it is
multiplication by $|\II_{ij}|^2$.

\begin{proposition}\label{prop:spectral-hessian}
For the preceding variations,
\begin{equation}\label{eq:spectral-hessian}
\begin{split}
 Q_\Sigma(\Phi,\Psi)=\sum_{i<j}\bigg[&
 \int_{S_{ij}}\bigl(\ip{\nabla^\perp\Phi_{ij}}{\nabla^\perp\Psi_{ij}}
       -\ip{\mathcal A_{ij}\Phi_{ij}}{\Psi_{ij}}\bigr)\,dA\\
 &-\int_{\sigma_{ij}}\ip{\Phi_{ij}}{\Psi_{ij}}\,ds
 -\sum_{k\in\{i,j\}}\int_{I_k}
        \kappa_{ij,k}\ip{\Phi_{ij}}{\Psi_{ij}}\,ds\bigg],
\end{split}
\end{equation}
where $\kappa_{ij,k}=\ip{D_TT}{\tau_{ij,k}}$ uses the outward
junction conormal. There is no separate term at the central vertex or at an outer
corner. The form has finite index and finite nullity.
\end{proposition}
\begin{proof}
It is enough to calculate on the diagonal and polarize. The interior
Jacobi term and the spherical free boundary term are standard. See
\cite{fraser2016sharp} and~\cite{tran2020index}. The
multiple-junction calculation in codimension one appears in
\cite{wang2022curvature}. We derive the vector junction
cancellation needed here and the treatment of the vertices.

Fix one junction, suppress its index, and let $W$ be the common variational vector field. With one unit tangent $T$, write
\[
 W=\Phi+d\tau+hT,\qquad K=D_TT,\qquad
 A=\II(T,T),\qquad B=\II(T,\tau).
\]
The Gauss formula and minimality give
\[
 A=K-\kappa\tau,\qquad B=D_T\tau+\kappa T,
 \qquad \II(\tau,\tau)=-A.
\]
Differentiating the transported conormal and $\ip{\tau}{T}=0$ yields
\[
 (\dot\tau)^\perp=\nabla^\perp_\tau\Phi-dA+hB,
 \qquad
 \ip{\dot\tau}{T}=\ip{\Phi}{B}-T(d)-h\kappa,
\]
and therefore
\begin{equation}\label{eq:spectral-dot-conormal}
 \ip{W}{\dot\tau}
 =\ip{\Phi}{\nabla^\perp_\tau\Phi}-d\ip{\Phi}{A}
   +2h\ip{\Phi}{B}-hT(d)-h^2\kappa.
\end{equation}
The conormal balance and its tangential derivative give $\sum d=\sum\kappa=0$ and $\sum B=0$. Since $B$ is normal, $\sum\ip{\Phi}{B}=0$, while
\[
 \ip{\Phi}{A}=\ip{W}{K}-\kappa d,
 \qquad |\Phi|^2=|W|^2-d^2-h^2.
\]
Thus
\[
 \sum d\ip{\Phi}{A}=-\sum\kappa d^2
 =\sum\kappa|\Phi|^2,
 \qquad T\!\left(\sum d\right)=0.
\]
Pointwise on the junction, this becomes
\begin{equation}\label{eq:spectral-junction-flux}
 \sum\ip{W}{\dot\tau}
 =\sum\left(\ip{\Phi}{\nabla^\perp_\tau\Phi}
                         -\kappa|\Phi|^2\right).
\end{equation}

Combining this with the standard outer boundary term
$\ip{\Phi}{\nabla^\perp_\eta\Phi}-|\Phi|^2$ and applying Green's
formula yields \eqref{eq:spectral-hessian} for regular variational fields.

To see that no additional term arises at a corner or at the central
vertex, excise corner neighborhoods of radius $\varepsilon$ in ordinary
face coordinates. The differentiated first-variation flux and the
normal Green flux on each artificial arc are bounded by $C$ times its
length, hence are $O(\varepsilon)$, because the metric and the
variational field are regular through the corner; letting
$\varepsilon\downarrow0$ proves the claim. Only first spatial
derivatives of the variational field enter, so the calculation extends
to the stated $C^{1,\alpha}$ fields by approximation.

Finally, bounded curvature coefficients and the trace inequality give
\[
 \|\Phi\|_{L^2(\partial S_{ij})}^2
 \leq\varepsilon\|\nabla^\perp\Phi\|_{L^2(S_{ij})}^2
       +C_\varepsilon\|\Phi\|_{L^2(S_{ij})}^2.
\]
The formula extends continuously to $\mathcal V_\Sigma$, and
$Q_\Sigma+C\|\Phi\|_{L^2}^2$ is coercive for large $C$. The form
domain contains all smooth normal fields compactly supported in face
interiors, and is therefore dense in the facewise $L^2$ space. Its
inclusion into that space is compact, so the associated self-adjoint
operator has compact resolvent. In particular, the index and nullity
are finite.
\end{proof}

\section{The normal form domain and angular operators of the flat cone}
\label{sec:spectral-angular}
For the standard cone, let $v_1,\ldots,v_4$ and $e_{ij}$ be the
vectors introduced in the flat model of Section~\ref{Sec1-Preliminaries}.
Put $V=\spanop\{v_1,\ldots,v_4\}$ and $F=V^\perp$. The space $V$
has dimension three. Choose an orientation of $V$, and let $\times$
denote the associated cross product on this oriented Euclidean
three-space. Choose an orthonormal basis $E_1,\ldots,E_{n-3}$ of $F$;
the list is empty when $n=3$.

Orient each of the six link edges from $i$ to $j$ with $i<j$. Write
$e=(i,j)$, $F_e=F_{ij}$, and $\sigma_e=F_e\cap\Sph^{n-1}$. Each face
is counted once in this orientation. Set
\[
 \omega_e(\theta)=\cos\theta\,v_i+\sin\theta\,e_{ij},\qquad
 x_e(r,\theta)=r\omega_e(\theta),\qquad
 \nu_e=\frac{v_i\times v_j}{\sin\vartheta}.
\]
Here $0\leq\theta\leq\vartheta$. The vector $\nu_e$ is the constant
unit normal to the face inside $V$, and
$\nu_e=\omega_e\times\omega_e'$. Since the tangent plane of $F_e$
is a two-dimensional subspace of $V$, its normal space in $\R^n$ is $\R\nu_e\oplus F$, so every normal field has a unique expression
\begin{equation}\label{eq:spectral-frame}
 \Phi_e=f_e\nu_e+\sum_{a=1}^{n-3}h_{ea}E_a.
\end{equation}
Define the incidence sign and endpoint coordinate by
\[
 s_{ke}=\begin{cases}-1,&k=i,\\+1,&k=j,\end{cases}
 \qquad
 \theta_{ke}=\begin{cases}0,&k=i,\\\vartheta,&k=j,\end{cases}
 \qquad(e=(i,j)).
\]
We only use these symbols when $e$ is incident to $k$.

\subsection{The trace conditions}
At $x=rv_k$ the outward face conormal is
$s_{ke}\omega_e'(\theta_{ke})$, and the consistently oriented
normal $N_{e,k}=v_k\times\tau_{e,k}$ equals $s_{ke}\nu_e$. The model
conormal balance proved in Section~\ref{Sec1-Preliminaries} gives
\[
 \sum_{e\ni k}s_{ke}\nu_e=0.
\]
The three vectors are unit vectors, so their pairwise angles are
$120^\circ$.

Let $W_k$ be the common variational vector field along the radial
junction $\Gamma_k^0$. Since $\nu_e$ and every $E_a$ are normal to $F_e$,
\[
 f_e(r,\theta_{ke})=\ip{W_k}{\nu_e},\qquad
 h_{ea}(r,\theta_{ke})=\ip{W_k}{E_a}.
\]
The preceding balance then gives
\begin{equation}\label{eq:spectral-traces}
 \sum_{e\ni k}s_{ke}f_e(r,\theta_{ke})=0,\qquad
 h_{ea}(r,\theta_{ke})=h_{da}(r,\theta_{kd})\quad(e,d\ni k).
\end{equation}
Let $\mathcal W$ be the closed subspace of the facewise $H^1$ normal
fields satisfying these relations in the trace sense on every open
radial junction. The trace identities show that
$\mathcal V_{TC}\subset\mathcal W$, and
Proposition~\ref{prop:spectral-form-domain} proves the reverse
inclusion, with no additional condition at the central vertex. Note the
difference between the two relations: the first is a signed incidence
condition, whereas in the additional normal directions the traces are
equal.
Reversing one oriented edge replaces $\theta$ by $\vartheta-\theta$
and $\nu_e$ by $-\nu_e$. Accordingly $f_e$ changes sign and is
reparametrized, while $h_{ea}$ is only reparametrized. The incidence
signs at its endpoints also reverse. Both trace conditions and the
quadratic form are unchanged under this unitary change of coefficients.

\subsection{Two angular operators}
Let $G$ be the metric graph $K_4$ whose edges are the intervals
$[0,\vartheta]$ with the orientations above. Define
\[
 \mathscr H=\bigoplus_e L^2(0,\vartheta),
\]
and the two closed subspaces of $\bigoplus_eH^1(0,\vartheta)$,
\begin{align*}
 \mathscr D_0&=\{h:h_e(\theta_{ke})\text{ is independent of }e\ni k
                         \text{ for each }k\},\\
 \mathscr D_1&=\{f:\sum_{e\ni k}s_{ke}f_e(\theta_{ke})=0
                         \text{ for each }k\}.
\end{align*}
The quadratic forms
\[
 \mathfrak a_0(h)=\sum_e\int_0^\vartheta |h_e'|^2,\qquad
 \mathfrak a_1(f)=\sum_e\int_0^\vartheta |f_e'|^2,
\]
with domains $\mathscr D_0$ and $\mathscr D_1$, respectively, are
densely defined, closed, and nonnegative. Define
$D:\mathscr D_0\subset\mathscr H\to\mathscr H$ by $(Dh)_e=h_e'$.
Integration by parts yields
\begin{equation}\label{eq:spectral-adjoint}
 \operatorname{dom}D^*=\mathscr D_1,\qquad (D^*f)_e=-f_e'.
\end{equation}
Indeed, the endpoint term is
$\sum_k h(k)\sum_{e\ni k}s_{ke}f_e(\theta_{ke})$. The common values
$h(k)$ can be chosen independently. Conversely, interior tests first
imply $f_e\in H^1$, and these endpoint tests give exactly the displayed
condition.

Let $\mathcal L_0=D^*D$ and $\mathcal L_1=DD^*$. They both act as
$-\partial_\theta^2$. Their operator domains are the edgewise $H^2$
functions with, respectively,
\begin{equation}\label{eq:spectral-operator-conditions}
\begin{array}{ll}
 \mathcal L_0:&h\in\mathscr D_0,\quad
          \displaystyle\sum_{e\ni k}s_{ke}h_e'(\theta_{ke})=0,\\[2mm]
 \mathcal L_1:&f\in\mathscr D_1,\quad
          f_e'(\theta_{ke})\text{ is independent of }e\ni k.
\end{array}
\end{equation}
For $\mathcal L_0$, the first line says that the function is continuous
at each graph vertex and that the sum of its outgoing derivatives is
zero there. For $\mathcal L_1$, it is the oriented derivatives $f_e'$
rather than the outward derivatives $s_{ke}f_e'$ that agree at a
vertex. These derivative relations belong to the operator domains, not
to the form domains. The internal normal sector uses $\mathcal L_1$,
and each direction in $F$ uses $\mathcal L_0$.

\begin{lemma}\label{lem:spectral-angular-pair}
The operators $\mathcal L_0$ and $\mathcal L_1$ have compact resolvent.
Their positive spectra agree, including multiplicities, and
\[
 \dim\ker\mathcal L_0=1,\qquad\dim\ker\mathcal L_1=3.
\]
The corresponding zero modes are the common constants and the
edgewise constants satisfying signed incidence balance.
\end{lemma}
\begin{proof}
The form domains $\mathscr D_0$ and $\mathscr D_1$ are closed
subspaces of the finite direct sum of interval $H^1$ spaces. Their
embeddings into $\mathscr H$ are compact, so the associated operators
have compact resolvent.

Since $\mathcal L_0=D^*D$ and $\mathcal L_1=DD^*$,
\[
 \ker\mathcal L_0=\ker D,\qquad
 \ker\mathcal L_1=\ker D^*.
\]
The first kernel consists of the functions that are constant on the
connected graph. For the second, $D^*f=0$ makes $f$ constant on each
edge, and the six edge constants form a vector in the kernel of the
$4$ by $6$ oriented incidence matrix $B$. Because $K_4$ is connected,
$\ker B^{\mathsf T}$ consists exactly of the constant vertex vectors, so $\operatorname{rank}B=3$, and rank-nullity gives
$\dim\ker B=6-3=3$.

For $\lambda>0$, the map $D/\sqrt\lambda$ takes the $\lambda$
eigenspace of $\mathcal L_0$ into that of $\mathcal L_1$. Indeed, if
$D^*Dh=\lambda h$, then $Dh\in\operatorname{dom}D^*$ and
$D^*Dh=\lambda h\in\operatorname{dom}D$, so
$DD^*(Dh)=\lambda Dh$. Moreover
$\|Dh\|^2=\lambda\|h\|^2$. The same argument with $D$ and $D^*$
interchanged shows that $D^*/\sqrt\lambda$ is the inverse map, and the positive spectra agree with multiplicity.
\end{proof}

\subsection{The point at the center and density of regular variations}
\begin{proposition}\label{prop:spectral-form-domain}
One has $\mathcal V_{TC}=\mathcal W$. There is no additional point
trace condition at the center in this $H^1$ form domain. In the
coefficients \eqref{eq:spectral-frame}, put
\[
 \mathscr H_*=\mathscr H\oplus\mathscr H^{\oplus(n-3)},\qquad
 \mathfrak a=\mathfrak a_1\oplus
              \mathfrak a_0^{\oplus(n-3)}.
\]
Then
\begin{equation}\label{eq:spectral-polar-form}
 Q_{TC}(\Phi,\Phi)=
 \int_0^1\bigl(r\|\partial_r\Phi\|_{\mathscr H_*}^2
                  +r^{-1}\mathfrak a(\Phi)\bigr)\,dr
                 -\|\Phi(1,\cdot)\|_{\mathscr H_*}^2.
\end{equation}
\end{proposition}
\begin{proof}
The flat second fundamental forms and junction curvatures vanish, so
Proposition~\ref{prop:spectral-hessian} gives
\eqref{eq:spectral-polar-form} on $\mathcal V_{TC}$. The same
expression defines a continuous form on $\mathcal W$. The trace
relations already give $\mathcal V_{TC}\subset\mathcal W$, so it
remains to prove the reverse inclusion.

\proofstep{1}{Approximate by finitely many angular modes.}
By Lemma~\ref{lem:spectral-angular-pair}, the operator
\[
 \mathcal L=\mathcal L_1\oplus\mathcal L_0^{\oplus(n-3)}
\]
has a complete real orthonormal eigenbasis $(\psi_m)$ in
$\mathscr H_*$, with eigenvalues $\lambda_m\geq0$. For
$\Phi\in\mathcal W$, write
$\Phi(r,\theta)=\sum_m a_m(r)\psi_m(\theta)$ in $\mathscr H_*$.
Fix $\delta>0$. On each annulus $\delta<r<1$, polar coordinates
identify the facewise $H^1$ space with the usual $H^1$ space on a
finite union of rectangles. Fubini's theorem shows that for almost
every $r$ the angular section belongs to the corresponding
one-dimensional $H^1$ space. The radial-edge trace relations in
\eqref{eq:spectral-traces} then hold for almost every such $r$, so
these angular sections lie in the form domain of $\mathcal L$.
The spectral resolution of the angular form and Parseval's identity
give
\begin{equation}\label{eq:spectral-mode-energy}
 \|\Phi\|_{H^1}^2
 =\sum_m\int_0^1\left(r|a_m|^2+r|a_m'|^2
                            +\frac{\lambda_m}{r}|a_m|^2\right)\,dr.
\end{equation}
Here the $H^1$ norm is the sum of the flat face norms. The identity
is first obtained on $\delta<r<1$ and then on $0<r<1$ by monotone
convergence as $\delta\downarrow0$. On compact radial subintervals,
$a_m$ is weakly $H^1$ and
$a_m'=\ip{\partial_r\Phi}{\psi_m}_{\mathscr H_*}$. Since the angular
projections do not depend on $r$, finite angular sums converge to
$\Phi$ in $H^1$. Their outer traces converge by the facewise trace
theorem. Each $\psi_m$ is edgewise smooth up to the two endpoints,
because it solves a constant-coefficient ordinary differential
equation there.

\proofstep{2}{Remove the central point without changing the closure.}
Fix a finite angular sum. Truncate each real radial coefficient at
levels $\pm M$. The Sobolev chain rule and dominated convergence in
the three nonnegative terms of \eqref{eq:spectral-mode-energy} show
that these truncations converge as $M\to\infty$, so we may assume
the finitely many radial coefficients are bounded.

For $0<\varepsilon<1/2$, multiply them by
\[
 \chi_\varepsilon(r)=
 \begin{cases}
 0,&0\leq r\leq\varepsilon^2,\\
 \log(r/\varepsilon^2)/\log(1/\varepsilon),
                                  &\varepsilon^2<r<\varepsilon,\\
 1,&\varepsilon\leq r\leq1.
 \end{cases}
\]
If $|a_m|\leq M$, then
\[
 \int_0^1r|\chi_\varepsilon'a_m|^2\,dr
 \leq M^2\int_{\varepsilon^2}^{\varepsilon}
          \frac{dr}{r\log^2(1/\varepsilon)}
 =\frac{M^2}{\log(1/\varepsilon)}\longrightarrow0.
\]
Moreover,
$(\chi_\varepsilon a_m-a_m)'=
(\chi_\varepsilon-1)a_m'+\chi_\varepsilon'a_m$.
The first term tends to zero in the weighted derivative norm because
$r|a_m'|^2$ is integrable and its support shrinks to the origin. The
estimate above controls the second term. The $L^2$ and angular terms
are bounded by the corresponding integrals over $(0,\varepsilon)$,
which also tend to zero. Taking the limits in the order of finite
angular truncation, scalar truncation, and then
$\varepsilon\downarrow0$ gives $H^1$ approximation by fields that
vanish near the center. The remaining radial coefficients can be
approximated by smooth functions while keeping a common zero
neighborhood of the center. We obtain facewise smooth compatible
normal fields with the trace relations \eqref{eq:spectral-traces}.

\proofstep{3}{Realize the approximating fields by admissible variations.}
For such a smooth compatible field, define on junction $k$
\begin{equation}\label{eq:spectral-flat-lift}
 W_k(r)=\frac23\sum_{e\ni k}f_e(r,\theta_{ke})\nu_e
                +\sum_a h_{ea}(r,\theta_{ke})E_a.
\end{equation}
The second sum is independent of the incident edge $e$. The Gram matrix of the three $\nu_e$ at this junction has entries $\frac32\delta_{ed}-\frac12s_{ke}s_{kd}$. For an incident edge $d$,
\[
 \ip{W_k}{\nu_d}
 =\frac23\left(\frac32 f_d
       -\frac12s_{kd}\sum_{e\ni k}s_{ke}f_e\right)=f_d,
\]
and the $E_a$ components of $W_k$ are the common $h_{ea}$ values, so $P_e^\perp W_k=\Phi_e$ on every radial edge.

Set $Y_e=W_k-\Phi_e$ there. Since $P_e^\perp W_k=\Phi_e$, this
field is tangent to the flat face. At an outer endpoint, both $W_k$
and $\Phi_e$ are perpendicular to the radial vector $v_k$, so $Y_e$
is tangent to the outer arc as well. Choose a smooth tangent field
along the outer arc with these two endpoint values. The resulting
boundary tangent fields agree at each corner and extend smoothly over
the flat face. For instance, in a corner chart, coefficient functions
$A(s)$ and $B(t)$ with common corner value $c$ extend locally as
$A(s)+B(t)-c$, and a partition of unity gives a global extension.
Because the normal fields vanish in a common neighborhood of the
center, the extensions can also be chosen zero there.

Now $U_e=\Phi_e+Y_e$ has one common value $W_k$ along each junction,
including the common value zero near the center, and
$\ip{x_e}{U_e}=0$ on the outer arc. Lemma~\ref{lem:spectral-full-field-lift}
therefore realizes $U$ as the variational vector field of an
admissible variation whose normal component is $\Phi$, so every approximating field lies in the defining class for $\mathcal V_{TC}$.
Passing to the $H^1$ limit proves $\mathcal W\subset\mathcal V_{TC}$.
\end{proof}

\section{The harmonic spectrum, index, and normal kernel}
\label{sec:spectral-model}
The angular calculation below is on the fixed flat cone and does not
involve matching of conformal parameters on a general $T$-surface.
Wang~\cite{wang2024networks} computed the length index and nullity of
embedded stationary triple junction networks in $\Sph^2$; for the
tetrahedral link these are three and three. The problem here is the
area Hessian on the entire cone, including its radial modes and the
additional normal directions, and we carry out the angular computation
in the incidence convention fixed above.

\subsection{The complete angular spectrum}
\begin{lemma}\label{lem:spectral-K4}
Write a positive angular eigenvalue as $k^2$. For $\mathcal L_0$ the
possible $k>0$ and their multiplicities are
\begin{equation}\label{eq:spectral-K4-list}
\begin{array}{c|c}
 k&\text{multiplicity}\\ \hline
 (2m\pi+\vartheta)/\vartheta,\quad m=0,1,2,\ldots&3\\
 (2m\pi-\vartheta)/\vartheta,\quad m=1,2,\ldots&3\\
 m\pi/\vartheta,\quad m=1,3,5,\ldots&2\\
 m\pi/\vartheta,\quad m=2,4,6,\ldots&4.
\end{array}
\end{equation}
The positive spectrum of $\mathcal L_1$ is the same. Their zero
multiplicities are one and three, respectively.
\end{lemma}
\begin{proof}
\proofstep{1}{The nonresonant case.}
Suppose $\sin(k\vartheta)\ne0$ and let $b_i$ be the common value
at vertex $i$ of an eigenfunction of $\mathcal L_0$. On edge $(i,j)$,
\[
 h_{ij}(\theta)=
 \frac{b_i\sin(k(\vartheta-\theta))+b_j\sin(k\theta)}
      {\sin(k\vartheta)}.
\]
Its outgoing derivative into the edge at vertex $i$ is
$k(b_j-b_i\cos(k\vartheta))/\sin(k\vartheta)$.
The zero-sum derivative conditions become
\begin{equation}\label{eq:spectral-adjacency}
 \cos(k\vartheta)b_i=\frac13\sum_{j\ne i}b_j.
\end{equation}
The matrix on the right is $(\mathbf1\mathbf1^{\mathsf T}-I)/3$.
It has eigenvalue $1$ on constants and $-1/3$ on the
three-dimensional space $\sum_i b_i=0$. The first would require
$\sin(k\vartheta)=0$, contrary to the present assumption. The
second gives $k\vartheta=2m\pi\pm\vartheta$, with the positive
values listed above. The edge solution is determined by the
vertex values, so every such multiplicity is exactly three.

\proofstep{2}{The resonant cases.}
Let $k\vartheta=m\pi$, $m\geq1$. The general solution on edge
$(i,j)$ is
\[
 h_{ij}(\theta)=b_i\cos(k\theta)+a_{ij}\sin(k\theta),\qquad
 b_j=(-1)^m b_i.
\]
The derivatives pointing into an edge at its tail and head are
respectively $k a_{ij}$ and $-(-1)^m k a_{ij}$. The sine term in
the derivative of the cosine vanishes at these endpoints.
If $m$ is odd, the endpoint relations around any triangle in $K_4$
imply that all $b_i=0$. Both outgoing derivatives on an edge then
have the sign of its coefficient $a_{ij}$, so the zero-sum
outgoing derivative condition at vertex $i$ is
\[
 \sum_{e\ni i} a_e=0.
\]
Let $C$ be the $4\times6$ unsigned incidence matrix of $K_4$.
If $z\in\ker C^{\mathsf T}$, then $z_i+z_j=0$ on every edge.
The presence of a triangle gives $z_i=0$ at every vertex, so
$\ker C^{\mathsf T}=\{0\}$ and $\operatorname{rank}C=4$, hence $\dim\ker C=2$, which is the multiplicity in the odd
resonant case.

If $m$ is even, all $b_i$ have a common value $b$. The sine
coefficients satisfy the equations for the oriented incidence
matrix $B$. Since $K_4$ is connected, $\operatorname{rank}B=3$,
so $\dim\ker B=3$. The common cosine coefficient $b$ supplies
one additional independent parameter, and the multiplicity is four.

Every eigenfunction has one of these edgewise ordinary differential
equation forms, so the list is exhaustive. The zero eigenvalue and
the transfer to $\mathcal L_1$ follow from
Lemma~\ref{lem:spectral-angular-pair}.
\end{proof}

\subsection{Harmonic extension and the Steklov form}
Put
\[
 \begin{aligned}
 E(\Phi,\Psi)&=\sum_e\int_{F_e}
 \ip{\nabla^\perp\Phi_e}{\nabla^\perp\Psi_e}\,dA,\\
 B(\Phi,\Psi)&=\sum_e\int_{\sigma_e}\ip{\Phi_e}{\Psi_e}\,ds.
 \end{aligned}
\]
Accordingly, $Q_{TC}=E-B$. A weak harmonic field is an element of
$\mathcal V_{TC}$ that is $E$-orthogonal to all elements with zero
outer trace. Equivalently, it is harmonic on each face and, in the
coefficient decomposition of Section~\ref{sec:spectral-angular}, satisfies
the two junction derivative conditions in
\eqref{eq:spectral-operator-conditions}. Proposition~\ref{prop:spectral-form-domain}
shows that there is no additional point condition at the center. A
harmonic Steklov eigenfield is a nonzero field satisfying
\begin{equation}\label{eq:spectral-Steklov-weak}
 E(\Phi,\Psi)=\sigma B(\Phi,\Psi)
                   \quad\text{for every }\Psi\in\mathcal V_{TC}.
\end{equation}
Here $\sigma$ is the Steklov boundary parameter: the field is
harmonic in the faces, and its outer normal derivative equals
$\sigma$ times its outer trace in the weak sense; it is not an
interior $L^2(TC)$ spectral parameter.

\begin{proposition}\label{prop:spectral-extension}
Let $(\psi_m)$ be a real orthonormal eigenbasis of the combined angular
operator $\mathcal L$ on $\mathscr H_*$, with
$\mathcal L\psi_m=k_m^2\psi_m$ and $k_m\geq0$.
A coefficient sequence $(b_m)$ is the outer-trace expansion of a
finite-energy weak harmonic field $H\in\mathcal V_{TC}$ precisely when
$\sum_m(1+k_m)|b_m|^2<\infty$. In that case the unique harmonic extension is
\begin{equation}\label{eq:spectral-extension}
 H(r,\theta)=\sum_m b_m r^{k_m}\psi_m(\theta).
\end{equation}
For every $\Phi\in\mathcal V_{TC}$ there is a unique decomposition
$\Phi=H+W$, where $H$ has the same outer trace and is weak harmonic,
and $W$ has zero outer trace. One has
\begin{equation}\label{eq:spectral-orthogonal-split}
 Q_{TC}(\Phi,\Phi)
 =\sum_m(k_m-1)|b_m|^2+E(W,W),
 \qquad E(W,W)>0\quad\text{if }W\ne0.
\end{equation}
The harmonic Steklov spectrum is exactly the set of $k_m$, with
its angular multiplicities. Under the identification of an outer trace with its coefficients, the
Dirichlet-to-Neumann operator is
\begin{equation}\label{eq:spectral-DtN}
 \Lambda_{TC}=\sqrt{\mathcal L}
 =\sqrt{\mathcal L_1}\oplus
                 (\sqrt{\mathcal L_0})^{\oplus(n-3)}.
\end{equation}
Its operator domain is the set of $b\in\mathscr H_*$ with
$\sum_m k_m^2|b_m|^2<\infty$, namely
$\mathscr D_1\oplus\mathscr D_0^{\oplus(n-3)}$.
The space of outer traces of the full $H^1$ form domain is instead
$\operatorname{dom}(\mathcal L^{1/4})$, whose squared graph norm is
equivalent to $\sum_m(1+k_m)|b_m|^2$. This distinction avoids imposing
pointwise endpoint values for arbitrary finite-energy boundary traces.
\end{proposition}
\begin{proof}
\proofstep{1}{Separate off the harmonic field.}
Fix the outer trace of $\Phi$ and minimize $E$ over the corresponding
affine subspace of $\mathcal V_{TC}$. It is nonempty by construction,
and continuity of the trace map makes it weakly closed. The following Poincar\'e estimate provides coercivity on the
subspace with zero outer trace. For a scalar component $w$, this trace
condition means $w(1,\theta)=0$ for almost every $\theta$, and weighted
Cauchy--Schwarz gives
\[
 |w(r,\theta)|^2
 \leq\log(1/r)\int_r^1 s|\partial_s w(s,\theta)|^2\,ds.
\]
Integrating in $r\,dr\,d\theta$ and using
$\int_0^1r\log(1/r)\,dr=1/4$ yields
\[
 \|W\|_{L^2(TC)}^2\leq\frac14
          \int_0^1r\|\partial_rW\|_{\mathscr H_*}^2\,dr
 \leq\frac14E(W,W).
\]
The estimate is first justified on annuli for Sobolev representatives
and then extended to the whole sector by monotone convergence. After subtracting one fixed extension of the trace, the displayed
estimate makes $E$ coercive on the zero-trace subspace. The direct
method therefore gives a unique minimizer $H$, whose Euler--Lagrange
equation is weak harmonicity.
For $W=\Phi-H$ it gives $E(H,W)=0$, and $B(H,W)=0$ as well.
The same Poincar\'e inequality yields $E(W,W)>0$ for nonzero $W$.

\proofstep{2}{Solve the radial equations and check finite energy.}
Expand $H$ in the angular basis and write its $m$th radial coefficient
as $a_m(r)$. Testing with $\eta(r)\psi_m(\theta)$ supported away
from $r=0,1$ gives
\[
 (r a_m')'-\frac{k_m^2}{r}a_m=0.
\]
For $k_m>0$, its solutions are $A r^{k_m}+C r^{-k_m}$.
The latter term has infinite energy and is excluded. For $k_m=0$,
the solutions are $A+C\log r$, and the logarithmic term again
has infinite energy, leaving $a_m(r)=b_m r^{k_m}$, including the
constant solution when $k_m=0$.
For an angular mode of unit $L^2$ norm,
\[
 E(r^k\psi,r^k\psi)=k,\qquad
 \|r^k\psi\|_{L^2(TC)}^2=\frac1{2(k+1)},\qquad
 B(r^k\psi,r^k\psi)=1.
\]
The first identity includes $k=0$. Parseval's identity yields
\eqref{eq:spectral-extension} and \eqref{eq:spectral-orthogonal-split}.
In particular,
\[
 \|H\|_{H^1(TC)}^2
 =\sum_m\left(k_m+\frac1{2(k_m+1)}\right)|b_m|^2.
\]
The weight in parentheses lies between $(1+k_m)/2$ and $1+k_m$.
We have proved both the trace criterion and the $H^1$ convergence
of the proposed extension whenever that criterion holds. For a finite sum, weak harmonicity and the boundary flux can be checked
mode by mode. For $r^k\psi$, integrate against a smooth compatible
normal test field $\Psi$ over $\varepsilon<r<1$. The angular operator-domain conditions in
\eqref{eq:spectral-operator-conditions} cancel the radial-edge fluxes. The absolute value of
the inner-arc flux is at most
\[
 k\varepsilon^k\|\psi\|_{\mathscr H_*}
                     \|\Psi(\varepsilon,\cdot)\|_{\mathscr H_*},
\]
which tends to zero for $k>0$, since $\Psi$ is bounded. It is
identically zero for $k=0$, and therefore
\[
 E(r^k\psi,\Psi)=k\ip{\psi}{\Psi(1,\cdot)}_{\mathscr H_*}.
\]
Proposition~\ref{prop:spectral-form-domain} gives density of these
smooth compatible tests, while energy and outer trace are continuous in $H^1$. The identity therefore extends to every form-domain test, and
passing to the $H^1$ limit proves that the proposed extension is the
required minimizer. Conversely every
outer trace of a form-domain field has such a minimizer, so the
trace-space statement is an equality of spaces.

\proofstep{3}{Identify all Steklov eigenfields.}
A field satisfying \eqref{eq:spectral-Steklov-weak} is harmonic by
testing with zero outer trace, so it has the expansion above.
Testing with each $r^{k_m}\psi_m$ yields
$(k_m-\sigma)b_m=0$. Conversely every such mode satisfies the weak
Steklov equation. Completeness of the angular basis shows that no other Steklov values
occur. The corresponding boundary operator is diagonal with entries
$k_m$.
The spectral definition of the square root gives
\eqref{eq:spectral-DtN}. Since $\operatorname{dom}\sqrt{\mathcal L}$
is the form domain of $\mathcal L$, it also gives the stated operator
domain. Finally, the eigenvalues of $\mathcal L$ are $k_m^2$, so the
squared graph norm of $\mathcal L^{1/4}$ has weight comparable to
$1+k_m$. This is exactly the trace criterion above.
\end{proof}

\begin{proof}[Proof of Theorem~\ref{thm:spectral-model-counts}]
\proofstep{1}{Count the Steklov modes, the index, and the kernel.}
The combined angular operator is one copy of $\mathcal L_1$ together
with $n-3$ copies of $\mathcal L_0$. Its zero eigenspace therefore has
dimension $3+(n-3)=n$. For every positive angular eigenvalue,
Lemma~\ref{lem:spectral-angular-pair} gives the same multiplicity for
$\mathcal L_0$ and $\mathcal L_1$, so the combined multiplicity is
$n-2$ times that in Lemma~\ref{lem:spectral-K4}. Since
$\pi/2<\vartheta<\pi$, the first positive values in
\eqref{eq:spectral-K4-list} are
\[
 1<\frac{\pi}{\vartheta}<\frac{2\pi}{\vartheta}-1.
\]
The first has multiplicity $3$ in either scalar angular problem, while
the second has multiplicity $2$, which gives
\[
 \sigma_0=0,\ m_0=n,\qquad
 \sigma_1=1,\ m_1=3(n-2),\qquad
 \sigma_2=\frac{\pi}{\vartheta},\ m_2=2(n-2).
\]

Let $\mathcal N_0$ be the $n$-dimensional zero-Steklov space. By
\eqref{eq:spectral-orthogonal-split}, $Q_{TC}$ is negative definite on
$\mathcal N_0$, so $\operatorname{ind}(TC)\ge n$. If $S$ is any
negative subspace, its projection onto $\mathcal N_0$ is injective.
Indeed, a nonzero field in the kernel of that projection would have
only modes with $k_m\ge1$, together with a zero-outer-trace remainder,
and would therefore have nonnegative quadratic value. It follows that
$\dim S\le n$ and $\operatorname{ind}(TC)=n$.

If $\Phi=H+W$ lies in the bilinear kernel, testing against the
zero-outer-trace field $W$ gives
$0=Q_{TC}(\Phi,W)=E(W,W)$, so $W=0$. Testing against each harmonic
Steklov mode then gives $(k_m-1)b_m=0$. Conversely, let $H_1$ be a harmonic mode with $k_m=1$ and decompose
an arbitrary test field as $H^\prime+W^\prime$. Weak harmonicity gives
$E(H_1,W^\prime)=0$, while the Steklov relation gives
$E(H_1,H^\prime)=B(H_1,H^\prime)$, so $Q_{TC}(H_1,\cdot)=0$.
The normal kernel is therefore exactly the $k=1$ eigenspace and has
dimension $3(n-2)$.

\proofstep{2}{Identify the zero Steklov fields.}
For $v\in V$, the constants $f_e=\ip{v}{\nu_e}$ satisfy the signed
incidence relations. The resulting map from $V$ to the corresponding
three-dimensional trace space is injective, since a vector orthogonal
to every $\nu_e$ lies in all six face planes and hence is zero. Since
both spaces have dimension three, the map is onto. The zero modes in
$F=V^\perp$ are the common constant coefficients. Hence all zero Steklov fields are
$(P_e^\perp v)_e$, $v\in\R^n$, as claimed.

\proofstep{3}{Identify the kernel fields as rotations.}
The $k=1$ eigenspace of $\mathcal L_0$ consists of
\[
 h_{a,e}(\theta)=\ip{a}{\omega_e(\theta)},\qquad a\in V.
\]
Indeed, these vertex values satisfy
$\sum_i\ip{a}{v_i}=\ip{a}{\sum_i v_i}=0$. The map from $V$ into
the three-dimensional space $\sum_i b_i=0$ is injective because the
$v_i$ span $V$, and hence it is onto. By the derivative isomorphism
of Lemma~\ref{lem:spectral-angular-pair}, the internal $k=1$
eigenspace is spanned by $-h_{a,e}'=-\ip{a}{\omega_e'}$.
For $\Omega_a x=a\times x$ on $V$ and zero on $F$,
\[
 \ip{\Omega_a x_e}{\nu_e}
 =r\ip{a}{\omega_e\times\nu_e}
 =-r\ip{a}{\omega_e'},
\]
so these rotations give all internal kernel fields. If $A:V\to F$ is
linear, then
\[
 \Omega_A=\begin{pmatrix}0&-A^*\\A&0\end{pmatrix}
\]
is skew-symmetric and satisfies $\Omega_Ax_e=rA\omega_e$. If
$E_1,\ldots,E_{n-3}$ is an orthonormal basis of $F$, then the choices
$A^*E_\alpha\in V$ are arbitrary. The mixed rotations therefore give
all $3(n-3)$ higher-codimension kernel fields. Rotations acting only
on $F$ vanish on the cone, so the entire normal kernel consists
of normal components of ambient rotations.
\end{proof}

\begin{corollary}\label{cor:spectral-first-gap}
Let $\Phi\in\mathcal V_{TC}$ have outer trace $L^2$-orthogonal to
the outer traces of the projected constants and of the rotational
kernel.
Then
\begin{equation}\label{eq:spectral-gap}
 Q_{TC}(\Phi,\Phi)\geq
 \left(\frac{\pi}{\vartheta}-1\right)B(\Phi,\Phi).
\end{equation}
The constant is sharp. Equality holds exactly when $\Phi$ belongs to
the harmonic Steklov eigenspace with value $\pi/\vartheta$.
\end{corollary}
\begin{proof}
Write $\Phi=H+W$ as in Proposition~\ref{prop:spectral-extension}.
Theorem~\ref{thm:spectral-model-counts} identifies the projected
constants and rotational kernel with the $k=0$ and $k=1$ harmonic
spaces, so the assumed orthogonality removes those coefficients.
Every remaining $k_m$ is at least $\pi/\vartheta$, so
\eqref{eq:spectral-orthogonal-split} gives
\[
 Q_{TC}(\Phi,\Phi)
 \geq\left(\frac{\pi}{\vartheta}-1\right)\sum_m|b_m|^2.
\]
Because $W$ has zero outer trace,
$B(\Phi,\Phi)=B(H,H)=\sum_m|b_m|^2$, proving the estimate. Equality
requires $W=0$ and all nonzero coefficients to lie at
$k_m=\pi/\vartheta$, which is exactly the stated eigenspace. Lemma~\ref{lem:spectral-K4} provides nonzero modes at this value.
They are $L^2$-orthogonal to the $k=0,1$ eigenspaces by
self-adjointness, so they satisfy the hypothesis and attain equality. This shows that the constant is sharp.
\end{proof}

\section{Projected translations and a sharp general index bound}
\label{sec:spectral-translations}
We return to a general regular free boundary minimal
$T$-surface with six faces. No conformal matching or disk topology is required
in this section. The identity below extends the smooth free boundary identity of Fraser and Schoen
\cite{fraser2016sharp} by checking both the junction
fluxes and the endpoints of the outer arcs. The analogous
three-face calculation appears in \cite{matinpour2025cone}.

\subsection{Admissibility of the projected constants}
Let $\mathcal V_{\mathrm{amb}}$ be the $H^1$ closure of normal fields
induced by smooth ambient vector fields tangent to the unit sphere, and
write $\operatorname{ind}_{\mathrm{amb}}(\Sigma)$ and
$\operatorname{nul}_{\mathrm{amb}}(\Sigma)$ for the index and the
bilinear-kernel dimension of $Q_\Sigma$ restricted to this subspace.
The flows of these ambient fields preserve the ball, the prescribed
junction identifications, and any additional coincidences of image
points, so $\mathcal V_{\mathrm{amb}}\subset\mathcal V_\Sigma$; the
two spaces are not identified here.

\begin{lemma}\label{lem:spectral-ambient-projections}
For every constant vector $v\in\R^n$, the normal family
$\Phi^v=(P_{ij}^\perp v)_{i<j}$ belongs to
$\mathcal V_{\mathrm{amb}}$.
\end{lemma}
\begin{proof}
Set $\rho(x)=1-|x|^2$. Choose a smooth function $\chi$ equal to one
on $(-\infty,1]$ and zero on $[2,\infty)$. For $\delta>0$ put
\[
 V_{v,\delta}(x)
 =v-\chi(\rho(x)/\delta)\ip{v}{x}x.
\]
Multiplying by a smooth ambient cutoff that is one near the closed
ball makes this field compactly supported. On the unit sphere its inner
product with $x$ vanishes, so its flow preserves the sphere and the ball
and is admissible. Its normal-component error on a face is
\[
 P_{ij}^\perp V_{v,\delta}-P_{ij}^\perp v
 =-\chi(\rho/\delta)\ip{v}{X_{ij}}X_{ij}^\perp.
\]
Along the outer edge,
$\partial_{-\eta_{ij}}(\rho\circ X_{ij})=2$. Together with the right
angles of Lemma~\ref{lem:outer-corners}, this makes $\rho\circ X_{ij}$
a boundary defining function up to the endpoints and shows that
$\{\rho\leq2\delta\}$ has area $O(\delta)$. Away from a fixed collar,
$\rho$ has a positive lower bound. Since $X_{ij}^\perp=0$ on the outer
edge and is $C^{1,\alpha}$,
\[
 |X_{ij}^\perp|\leq C\rho,\qquad
 |\nabla^\perp X_{ij}^\perp|\leq C
\]
in the collar; these are estimates on the abstract face, unaffected
by image self-intersections. Differentiating the cutoff costs
$1/\delta$ only
where $\delta\leq\rho\leq2\delta$, where
$X_{ij}^\perp=O(\delta)$, so the error is $O(|v|\delta)$, its first
derivative is $O(|v|)$, and its squared $H^1$ norm is
$O(|v|^2\delta)$. Summing over the faces proves convergence in
$\mathcal V_{\mathrm{amb}}$. The central vertex is outside the collar
for small $\delta$.
\end{proof}

\begin{theorem}\label{thm:spectral-translation-identity}
For $v,w\in\R^n$ constant and $\Phi^v_{ij}=P_{ij}^\perp v$,
\begin{equation}\label{eq:spectral-translation-identity}
 Q_\Sigma(\Phi^v,\Phi^w)
 =-2\sum_{i<j}\int_{S_{ij}}
             \ip{P_{ij}^\perp v}{P_{ij}^\perp w}\,dA.
\end{equation}
\end{theorem}
\begin{proof}
By Lemma~\ref{lem:spectral-ambient-projections} these fields lie in
the form domain. We prove the diagonal identity and polarize. On each
face a projected constant is a Jacobi field, $L_{ij}\Phi^v_{ij}=0$ in
the interior, because translation preserves the interior minimality
equation (though not the spherical boundary). Green's formula reduces
the Hessian to its edge fluxes.

\proofstep{1}{Cancel each junction flux.}
Fix one junction and suppress its indices. On each incident face write
\[
 v=\Phi+d\tau+hT,\qquad
 A=\II(T,T),\qquad B=\II(T,\tau),
\]
where $d=\ip{v}{\tau}$ and $h=\ip{v}{T}$. The normal part of
$D_\tau v=0$ gives
\[
 \nabla^\perp_\tau\Phi=dA-hB.
\]
Proposition~\ref{prop:spectral-hessian} gives $\sum B=0$ and
$\sum d\ip{\Phi}{A}=\sum\kappa|\Phi|^2$. Since $h$ is common to
the three faces,
\[
 \sum\left(\ip{\Phi}{\nabla^\perp_\tau\Phi}
                             -\kappa|\Phi|^2\right)=0.
\]
This cancellation holds pointwise on each of the four junctions.

\proofstep{2}{Keep and cancel the outer-arc endpoint terms.}
On an outer arc parametrized by arclength with unit tangent $Z$,
put $f=\ip{v}{X}$ and $\beta=\ip{v}{Z}$. Since $(X,Z)$ is an
orthonormal tangent frame and $\II(X,Z)=0$, one has
\[
 D_ZZ=-X+\II(Z,Z),\qquad f'=\beta,\qquad
 \beta'=-f+\ip{\Phi}{\II(Z,Z)},
 \qquad \nabla^\perp_\eta\Phi=f\II(Z,Z).
\]
The last identity is the normal part of $D_\eta v=0$, using
$\II(\eta,Z)=0$ and $\II(\eta,\eta)=-\II(Z,Z)$. Integration by
parts gives
\begin{align*}
 \int_{\sigma}\bigl(\ip{\Phi}{\nabla^\perp_\eta\Phi}-|\Phi|^2\bigr)ds
 &=\int_{\sigma}(f\beta'+f^2-|\Phi|^2)ds\\
 &=[f\beta]_{\partial\sigma}
          +\int_{\sigma}(2f^2-|v|^2)ds.
\end{align*}
At an outer endpoint $q_i$, the outward unit tangent of the outer
boundary interval is $\tau_{ij,i}$. This follows from the right angle
corner geometry and the convention that $\tau_{ij,i}$ points out of
the face across the junction edge. The three endpoint contributions
at this abstract vertex therefore sum to
\[
 \ip{v}{q_i}\ip{v}{\sum_{j\ne i}\tau_{ij,i}}=0.
\]
Distinct abstract endpoints are counted separately even if their
images coincide. We obtain
\begin{equation}\label{eq:spectral-outer-flux}
 Q_\Sigma(\Phi^v,\Phi^v)
 =\sum_{i<j}\int_{\sigma_{ij}}
                      (2\ip{v}{X_{ij}}^2-|v|^2)\,ds.
\end{equation}

\proofstep{3}{Convert the flux to an area integral.}
For the ambient vector field
$\mathcal Y_v(x)=|v|^2x-2\ip{v}{x}v$,
\[
 \operatorname{div}_{\Sigma_{ij}}\mathcal Y_v
 =2|v|^2-2|v_{ij}^{\top}|^2=2|P_{ij}^\perp v|^2.
\]
Because each face is minimal, the tangential divergence theorem
identifies the integral of this divergence with the conormal boundary
flux. The summed junction fluxes vanish by conormal balance. On the
sphere the flux is $|v|^2-2\ip{v}{X_{ij}}^2$, which is the negative of
\eqref{eq:spectral-outer-flux}. This proves the diagonal formula.
All edge quantities used above are continuous for the stated
regularity. Excision at the central vertex and outer corners adds
bounded fluxes over arcs whose lengths tend to zero. The interval
endpoint terms in Step~2 are retained before taking this limit.
Polarization completes the proof.
\end{proof}

\begin{corollary}\label{cor:spectral-general-index}
Every regular free boundary minimal $T$-surface with six faces satisfies
\[
 \operatorname{ind}(\Sigma)\geq n.
\]
The same lower bound holds for $\operatorname{ind}_{\mathrm{amb}}$
computed on $\mathcal V_{\mathrm{amb}}$. Both bounds are sharp,
since the standard cone has index $n$ in both variation spaces.
For the standard cone, the ambient normal nullity is also
$3(n-2)$, and its ambient normal kernel consists exactly of the normal
components of ambient infinitesimal rotations.
\end{corollary}
\begin{proof}
The map $\mathcal P:v\mapsto(P_{ij}^\perp v)_{i<j}$ is injective.
If all its normal components vanish in $H^1$, their continuous
representatives vanish also at the central vertex. Then $v$ lies
in every tangent plane there. Those six planes are
$\spanop\{t_i,t_j\}$, whose intersection is zero, as in the proof
of Lemma~\ref{lem:planar-T}, so $v=0$.
Theorem~\ref{thm:spectral-translation-identity} makes its
$n$-dimensional image negative definite. That image lies already
in $\mathcal V_{\mathrm{amb}}$, so both lower bounds follow.
The full-space model count is Theorem~\ref{thm:spectral-model-counts}.
Restriction cannot increase the index, and the projected constants
attain the same value in the ambient space.

For the model, \eqref{eq:spectral-orthogonal-split} gives a
$Q$-orthogonal decomposition into the negative projected constants,
the rotational kernel, and a strictly positive remainder. The first
two spaces lie in $\mathcal V_{\mathrm{amb}}$. If $\Phi$ belongs to
the restricted kernel, write $\Phi=\Phi_-+\Phi_0+\Phi_+$
accordingly. Testing against $\Phi_-$ gives
$0=Q_{TC}(\Phi,\Phi_-)=Q_{TC}(\Phi_-,\Phi_-)$, so $\Phi_-=0$.
Testing next against $\Phi$ itself gives
$Q_{TC}(\Phi_+,\Phi_+)=0$, and strict positivity implies
$\Phi_+=0$, leaving only $\Phi_0$. Conversely, every normal
component of an ambient infinitesimal rotation lies in both spaces and
in the full kernel, so it also belongs to the restricted kernel.
\end{proof}

\begin{remark}
Projected translations have rank $n$ on every regular $T$-surface,
including the flat cone, so this argument alone cannot characterize
equality in Corollary~\ref{cor:spectral-general-index}. The dilation
field developed in Section~\ref{sec:dilation-rigidity} supplies the
additional direction used for minimum-index rigidity in
Section~\ref{sec:minimum-index-rigidity}.
\end{remark}

\section{The dilation field and Dirichlet rigidity}\label{sec:dilation-rigidity}
Throughout this section, $\Sigma$ is a regular free boundary minimal
$T$-surface with six faces as in Definition~\ref{def:regular-T}, with
the normal variation space and area form of
Section~\ref{sec:spectral-variations}; no conformal matching or disk
topology is assumed. The Dirichlet condition below is imposed only on
the outer arcs, while the junctions remain coupled and are allowed to
move.

The smooth hypersurface analogue is the support function
$\langle X,N\rangle$; see \cite{tran2020index}.
Here we use the normal-vector-valued field $X^\perp$ and verify the
junction and corner properties needed for index and local rigidity.

\subsection{The Dirichlet form on the outer boundary}
Define the closed subspace
\[
 \mathcal V_\Sigma^D
 =\{\Phi\in\mathcal V_\Sigma:
       \Phi_{ij}|_{\sigma_{ij}}=0\text{ in the trace sense for every }i<j\},
 \qquad Q_\Sigma^D=Q_\Sigma|_{\mathcal V_\Sigma^D}.
\]
The junction trace conditions are those inherited from
$\mathcal V_\Sigma$, not separate zero conditions on the junctions.
The outer trace map is continuous from the facewise $H^1$ space to
$H^{1/2}$ on the outer arcs, so $\mathcal V_\Sigma^D$ is a closed
subspace of $\mathcal V_\Sigma$.
We define $\ker Q_\Sigma^D$ by testing against every field in
$\mathcal V_\Sigma^D$, and set
\begin{equation}\label{eq:dilation-first-Dirichlet}
 \lambda_1^D(\Sigma)
 =\inf_{0\ne\Phi\in\mathcal V_\Sigma^D}
       \frac{Q_\Sigma(\Phi,\Phi)}{\sum_{i<j}\int_{S_{ij}}|\Phi_{ij}|^2\,dA}.
\end{equation}
The bounded coefficients, trace inequality, and compact inclusion
used in Proposition~\ref{prop:spectral-hessian} apply to this
closed subspace as well. It contains the smooth normal fields
supported in face interiors, so it is dense in the facewise $L^2$
space. Its associated operator is self-adjoint with compact
resolvent. In particular, the infimum in
\eqref{eq:dilation-first-Dirichlet} is attained.

This space need not be identified with the variations fixing the
parametrized outer arcs; the field below belongs to that smaller
closure, even for ambient variations fixing a neighborhood of the
sphere.

\begin{proposition}\label{prop:dilation-field}
Set
\[
 \Xi_{ij}=P_{ij}^\perp X_{ij}.
\]
Then $\Xi\in\mathcal V_\Sigma^D\cap\mathcal V_{\mathrm{amb}}$ and
\begin{equation}\label{eq:dilation-Dirichlet-kernel}
 Q_\Sigma(\Xi,\Psi)=0\qquad(\Psi\in\mathcal V_\Sigma^D).
\end{equation}
For an arbitrary $\Psi\in\mathcal V_\Sigma$, the more precise identity is
\begin{equation}\label{eq:dilation-boundary-flux}
 Q_\Sigma(\Xi,\Psi)
 =\sum_{i<j}\int_{\sigma_{ij}}
       \ip{b_{ij}}{\Psi_{ij}}\,ds,
 \qquad b_{ij}=-\II_{ij}(\eta_{ij},\eta_{ij})
              =\II_{ij}(Z_{ij},Z_{ij}),
\end{equation}
where $Z_{ij}$ is either choice of unit tangent to the outer arc.
For every constant $v\in\R^n$, one also has
\begin{equation}\label{eq:dilation-translation-orthogonal}
 Q_\Sigma(\Xi,P^\perp v)=0.
\end{equation}
\end{proposition}
\begin{proof}
\proofstep{1}{Put the field in the form domain.}
The field $\Xi$ is facewise $C^{1,\alpha}$, so it belongs to
$H^1$. On an outer arc, $X=\eta$ is tangent to the face, so
$\Xi=0$, including at the endpoints by continuity.
For admissibility, use $\rho(x)=1-|x|^2$ and a smooth scalar cutoff
$\chi$ equal to one on $(-\infty,1]$ and zero on $[2,\infty)$.
For $\delta>0$ define the ambient field
\[
 V_\delta(x)=(1-\chi(\rho(x)/\delta))x,
\]
with a further smooth compactly supported cutoff equal to one
near the closed ball. This field vanishes in a neighborhood of the sphere. For sufficiently
small positive and negative times, its flow fixes a collar of the sphere
pointwise and preserves the ball. It also preserves every prescribed
junction identification. Its normal component along the immersed faces is
$(1-\chi(\rho/\delta))\Xi$.

The outer collar estimates used in
Lemma~\ref{lem:spectral-full-field-lift} hold also at the two
right angle corners of each outer arc. In those coordinates
$\rho\circ X$ is a regular boundary defining function, hence the set
$0\leq\rho\leq2\delta$ on the abstract faces has area $O(\delta)$.
Since $\Xi$ is $C^1$ and has zero trace on the outer edge,
\[
 |\Xi|\leq C\operatorname{dist}(\,\cdot\,,\sigma)
          \leq C'\rho,\qquad |\nabla^\perp\Xi|\leq C.
\]
Moreover $|\nabla(\rho\circ X)|$ is bounded. On the derivative support
of $\chi(\rho/\delta)$ one has $\rho=O(\delta)$, so the factor
$\delta^{-1}\Xi\,\nabla(\rho\circ X)$ is uniformly bounded. It follows that
\[
 \|\chi(\rho/\delta)\Xi\|_{H^1(\Sigma)}^2\leq C\delta\longrightarrow0.
\]
Their normal components converge to $\Xi$ in $H^1$, so
$\Xi\in\mathcal V_{\mathrm{amb}}$. Since its outer trace is zero,
$\Xi\in\mathcal V_\Sigma^D$ as well. Literal dilation need not
preserve the ball and is not used as a global admissible variation.

\proofstep{2}{Check the interior and junction equations.}
Euclidean dilations of a minimal face are minimal. Differentiating
this interior equation gives
$L_{ij}\Xi_{ij}=0$ in the face interior, where $L_{ij}$ is the
normal Jacobi operator of Section~\ref{sec:spectral-variations}.
This is a local equation and does not differentiate a boundary
constraint. Translations likewise give $L_{ij}(P_{ij}^\perp v)=0$.

The following junction computation applies to both fields. Let
$Y=a+\lambda X$, with $a\in\R^n$ and $\lambda\in\R$ constant,
and write $\Phi=P^\perp Y$. Suppress the junction index and use its
common unit tangent $T$. On the three incident faces put
\[
 Y=\Phi+d\tau+hT,\quad
 K=D_TT,\quad A=\II(T,T),\quad B=\II(T,\tau),
 \quad\kappa=\ip{K}{\tau}.
\]
The value of $Y$, and therefore $h$, is common to the three faces.
Since $D_ZY=\lambda Z$ is tangent for every face tangent $Z$, and
$Y^\top=d\tau+hT$, taking the normal part with $Z=\tau$ gives
\[
 \nabla^\perp_\tau\Phi=-\II(\tau,Y^\top)=dA-hB,
\]
where minimality was used in the form $\II(\tau,\tau)=-A$.
Use $A=K-\kappa\tau$, $B=D_T\tau+\kappa T$, and
$\sum\tau=0$. Because the same ambient vector $Y$ occurs on all
three sheets,
\[
 \sum d=\ip{Y}{\sum\tau}=0,
 \qquad \sum\kappa=\ip{K}{\sum\tau}=0.
\]
Differentiating the conormal balance along the common unit tangent gives
\[
 \sum D_T\tau=0,
 \qquad \sum B=0.
\]
Hence
\begin{equation}\label{eq:dilation-junction-vector}
 \begin{split}
 \sum\bigl(\nabla^\perp_\tau\Phi-\kappa\Phi\bigr)
 &=\sum\bigl(dK-hB-\kappa Y+\kappa hT\bigr)=0.
 \end{split}
\end{equation}
This is a vector identity in the common ambient space. It proves
the coupled conormal equation, not just cancellation of one
quadratic boundary value.

\proofstep{3}{Compute the remaining flux.}
On an outer arc, $X=\eta$, $\Xi=0$, and $D_\eta X=\eta$ is
tangent to the face. Taking normal parts gives
\[
 \nabla^\perp_\eta\Xi=-\II(\eta,X^\top)=-\II(\eta,\eta)=b.
\]
Minimality gives the alternative expression $b=\II(Z,Z)$.
Apply Green's formula first against a regular admissible normal
test $\Psi$. The interior term vanishes. On a junction,
$\Psi_{ij}=P_{ij}^\perp W_i$ for a common variational vector field $W_i$.
Writing
$F_{ij}=\nabla^\perp_{\tau_{ij}}\Xi_{ij}-\kappa_{ij}\Xi_{ij}$,
each $F_{ij}$ is normal to its own face, so
\[
 \sum_{j\ne i}\ip{F_{ij}}{\Psi_{ij}}
 =\sum_{j\ne i}\ip{F_{ij}}{W_i}
 =\ip{\sum_{j\ne i}F_{ij}}{W_i}=0
\]
by \eqref{eq:dilation-junction-vector}. The outer term $-\ip{\Xi}{\Psi}$ is zero, leaving
\eqref{eq:dilation-boundary-flux}.

The normal derivatives of $\Xi$ are bounded through each corner.
Excision of a corner creates arcs of length
$O(\varepsilon)$ with bounded Green flux against regular tests.
Their contributions tend to zero. There is no integration by parts
along a junction or an outer interval in this calculation, and no
endpoint term is discarded. Continuity of $Q_\Sigma$ and of the
outer trace, and boundedness of $b$, extend the identity to every
$\Psi\in\mathcal V_\Sigma$. In particular it is zero for all tests
in $\mathcal V_\Sigma^D$.

Finally apply the same Green calculation to the translation field
$P^\perp v$. Its interior equation and vector junction equation
were proved in Step~2. Pair it with $\Xi$, whose outer trace is
zero. All boundary contributions vanish, giving
\eqref{eq:dilation-translation-orthogonal} by symmetry. This last
step does not infer orthogonality from a zero quadratic value.
\end{proof}

\begin{lemma}\label{lem:dilation-zero-flat}
Each of the following conditions implies that $\Sigma$ is the standard
flat $T$-cone, up to an orthogonal transformation, with each face
covered once:
\begin{enumerate}[label=\textup{(\roman*)},leftmargin=2.4em]
 \item $\Xi\equiv0$;
 \item $b_{ij}\equiv0$ on every open outer arc.
\end{enumerate}
\end{lemma}
\begin{proof}
For (i), $X$ is tangent on each face. Differentiating yields
$\II(Z,X)=(D_ZX)^\perp=0$ for every tangent vector $Z$.
At a point where $X\ne0$, take an orthonormal tangent frame with
$e_1=X/|X|$. The identity above then gives
\[
 \II(e_1,e_1)=\II(e_1,e_2)=0,
 \qquad\II(e_2,e_2)=-\II(e_1,e_1)=0.
\]
Here the last equality uses both minimality and the fact that the
faces are two-dimensional. The zero set of $X$ has empty interior, since an open set mapped to the
origin would contradict the rank-two assumption. Hence $X\ne0$ on a
dense set, and continuity gives $\II=0$ everywhere. A connected face
with zero second fundamental form lies in an affine two-plane, so all faces are
planar. Lemma~\ref{lem:planar-T} identifies their images and multiplicity.

For (ii), the spherical free boundary identity yields
$\II(Z,\eta)=0$. The hypothesis and minimality give
$\II(\eta,\eta)=\II(Z,Z)=0$ as well. The whole second
fundamental form is zero along each open outer arc.
Lemma~\ref{lem:cauchy-plane} makes each connected face planar,
and Lemma~\ref{lem:planar-T} again applies.
\end{proof}

\begin{theorem}\label{thm:Dirichlet-rigidity}
For a regular free boundary minimal $T$-surface with six faces in
$\overline{\B^n}$, the following conditions are equivalent:
\begin{enumerate}[label=\textup{(\roman*)},leftmargin=2.4em]
 \item $\Sigma$ is the standard $T$-cone up to an orthogonal
       transformation, with every face covered once;
 \item $\ker Q_\Sigma^D=\{0\}$;
 \item $\lambda_1^D(\Sigma)>0$.
\end{enumerate}
In particular, on every nonflat surface in this class,
$\Xi=X^\perp$ is a nonzero Dirichlet kernel field and
$\lambda_1^D(\Sigma)\leq0$.
\end{theorem}
\begin{proof}
For (i)$\Rightarrow$(iii), apply an orthogonal transformation and take
$\Sigma=TC$. The Dirichlet form is then
$\sum_e\int_{F_e}|\nabla\Phi_e|^2$. For a scalar component $w$
with zero outer trace, weighted Cauchy--Schwarz yields
\[
 |w(r,\theta)|^2
 \leq\log(1/r)\int_r^1s|\partial_sw(s,\theta)|^2\,ds.
\]
Integrating in $r$ and applying Fubini, with
$\int_0^s r\log(1/r)\,dr\leq1/4$, and then integrating in $\theta$
and summing the normal components yields
\begin{equation}\label{eq:dilation-flat-Poincare}
 \|\Phi\|_{L^2(TC)}^2\leq\frac14
           \sum_e\int_{F_e}|\nabla\Phi_e|^2.
\end{equation}
The estimate is first applied on annuli to Sobolev representatives and
then passed to the limit at the center. It gives
$\lambda_1^D(TC)\geq4>0$ (a convenient rather than sharp bound). If
(iii) holds, a Dirichlet-kernel field has zero quadratic value, so the
Rayleigh quotient gives (ii). Conversely,
Proposition~\ref{prop:dilation-field} places $\Xi$ in
$\ker Q_\Sigma^D$ for every surface in the class, so (ii) gives
$\Xi=0$, and Lemma~\ref{lem:dilation-zero-flat} proves (i). It is
this rigidity statement, and not a general spectral principle, that
makes a trivial Dirichlet kernel imply positivity here. For a nonflat
surface, the same proposition and lemma make $\Xi$ a nonzero field
with zero Rayleigh quotient, so $\lambda_1^D(\Sigma)\leq0$.
\end{proof}

\section{Rigidity at minimum index}
\label{sec:minimum-index-rigidity}
The translation space from Section~\ref{sec:spectral-translations}
has dimension $n$ even on the flat $T$-cone, and the dilation field
gives the additional direction needed when the surface is not flat.
This is analogous to the fixed-boundary contribution in the smooth
index decomposition of \cite{tran2020index}; we give a direct argument
for the coupled vector-normal form.

\begin{lemma}\label{lem:dilation-negative-extension}
Let $Q$ be a symmetric bilinear form on a real vector space, and
let $\mathcal N$ be a $d$-dimensional negative-definite subspace.
Suppose $0\ne w\in\mathcal N^{\perp_Q}$ satisfies $Q(w,w)=0$
and $Q(w,z)\ne0$ for some $z$. Then the index of $Q$ is at least
$d+1$.
\end{lemma}
\begin{proof}
Since $Q|_{\mathcal N}$ is negative definite, it is nondegenerate. Choose
$n_z\in\mathcal N$ so that $z_0=z-n_z\in\mathcal N^{\perp_Q}$. Then
$c:=Q(w,z_0)=Q(w,z)\ne0$, and
\[
 Q(w+t z_0,w+t z_0)=2tc+t^2Q(z_0,z_0).
\]
For sufficiently small nonzero $t$ with sign opposite to $c$, this is
negative. Also $w+t z_0\in\mathcal N^{\perp_Q}$ and is nonzero,
since $c\ne0$ makes $w$ and $z_0$ linearly independent. Therefore
$\mathcal N\oplus\mathbb R(w+t z_0)$ is a $(d+1)$-dimensional
negative-definite subspace.
\end{proof}

\begin{proof}[Proof of Theorem~\ref{thm:minimum-index-rigidity}]
Let $\mathcal N=\{P^\perp v:v\in\R^n\}$. By
Theorem~\ref{thm:spectral-translation-identity} and
Corollary~\ref{cor:spectral-general-index}, $\mathcal N$ is an
$n$-dimensional negative-definite subspace of $\mathcal V_\Sigma$.
Suppose $\Sigma$ is not the standard flat $T$-cone. By
Lemma~\ref{lem:dilation-zero-flat}, $\Xi\ne0$ and at least one outer
curvature field $b_{ij}$ is not identically zero. Proposition~\ref{prop:dilation-field}
gives
\[
 Q_\Sigma(\Xi,\Xi)=0,\qquad Q_\Sigma(\Xi,\mathcal N)=0.
\]

Choose an outer arc $\sigma_{ij}$ for which $b:=b_{ij}$ is not
identically zero, and choose an interior point $q$ with $b(q)\ne0$.
Set $a=b(q)$. On a sufficiently small subarc, continuity gives
$\ip{b}{a}>0$. Choose a nonnegative $C^{1,\alpha}$ cutoff $\chi$ on
the corresponding abstract face, supported in a collar of this subarc
and away from the junction edges and outer corners. Define the normal family by
$\Psi=\chi P_{ij}^\perp a$ on this face and $\Psi=0$ on the other
five faces. Its junction trace is zero, and on the outer edge
$\ip{\Psi}{X}=0$ because $X=\eta$ is tangent to the face while $\Psi$
is normal. Lemma~\ref{lem:spectral-full-field-lift} therefore lifts
$\Psi$ to an admissible variational field, so its normal part belongs to
$\mathcal V_\Sigma$. Formula \eqref{eq:dilation-boundary-flux} yields
\[
 Q_\Sigma(\Xi,\Psi)
 =\int_{\sigma_{ij}}\chi\ip{b}{a}\,ds>0.
\]
The construction takes place on the abstract face, so additional
coincidences of image points do not matter. Lemma~\ref{lem:dilation-negative-extension}
now gives $\operatorname{ind}(\Sigma)\ge n+1$.

Every nonflat surface therefore has index at least $n+1$, while
Theorem~\ref{thm:spectral-model-counts} gives index $n$ for the
standard cone. This proves both directions of the characterization and
the stated lower bound for nonflat surfaces.
\end{proof}

\subsection{Ambient variations}
We give the additional boundary argument needed when test fields
must come from a single ambient vector field, even if the outer
arcs have self-intersections or coincide in part.

\begin{lemma}\label{lem:dilation-ambient-detection}
Let $b_{ij}$ be the outer curvature fields from
Proposition~\ref{prop:dilation-field}. If at least one $b_{ij}$ is not identically zero on its open outer
arc, there is a smooth ambient vector field
$V$, tangent to the unit sphere, such that
\[
 \sum_{i<j}\int_{\sigma_{ij}}\ip{b_{ij}}{V(X_{ij})}\,ds>0.
\]
\end{lemma}
\begin{proof}
An outer arc has a regular $C^{2,\alpha}$ arclength
parametrization $c$ in the sphere. Since the face has conormal
$c$ there, its curvature decomposition is
\begin{equation}\label{eq:dilation-boundary-geodesic-curvature}
 c''=-c+\II(Z,Z),\qquad b=c''+c.
\end{equation}
The vector $b$ is the geodesic curvature vector of the spherical curve.
It is unchanged by reversal of its arclength orientation.

Local injectivity of each regular $C^{2,\alpha}$ outer arc gives a
finite subdivision into compact embedded pieces. We keep coincident
pieces as distinct parametrized copies.

Choose an open subarc $J$ of one piece on which $b\ne0$, avoiding the
subdivision endpoints. For each other embedded piece $C$, either
$J\cap C$ has empty relative interior, in which case shrink $J$ into
$J\setminus C$, or it has nonempty relative interior, in which case
shrink $J$ into that interior. After finitely many such steps, $J$ is
either disjoint from or contained in the image of every piece.

Choose $q\in J$. A piece disjoint from $J$ has positive distance from
$q$, while a piece containing $J$ agrees near $q$ with the same
spherical subarc. Since there are finitely many pieces, an ambient
neighborhood $U$ of $q$ meets only coincident copies of this subarc.
Reversing arclength orientation leaves $c''+c$ unchanged, so
\eqref{eq:dilation-boundary-geodesic-curvature} gives the same vector
$b$ on every copy. Their contributions therefore add rather than cancel.

Put $a=b(q)$ at the chosen point and shrink $U$ so that
$\ip{b}{a}>0$ on its curve pieces. For a smooth nonnegative
ambient cutoff $\chi$ supported in $U$ and positive at $q$, set
\[
 V(x)=\chi(x)\bigl(a-\ip{a}{x}x\bigr).
\]
This is tangent to the unit sphere. Since $b\perp X$, its pairing
with $b$ on each copy is $\chi\ip{b}{a}>0$ where $\chi>0$.
The sum of integrals is positive.
\end{proof}

\begin{corollary}\label{cor:minimum-ambient-index-rigidity}
For ambient variations, the same minimum index characterization
holds:
\[
 \operatorname{ind}_{\mathrm{amb}}(\Sigma)=n
 \quad\Longleftrightarrow\quad
 \Sigma\text{ is the standard flat }T\text{-cone, up to }\mathrm O(n),
\]
with each face covered once. A nonflat surface has
$\operatorname{ind}_{\mathrm{amb}}(\Sigma)\geq n+1$.
\end{corollary}
\begin{proof}
Both $\mathcal N$ and $\Xi$ belong to $\mathcal V_{\mathrm{amb}}$
by Lemma~\ref{lem:spectral-ambient-projections} and
Proposition~\ref{prop:dilation-field}. On a nonflat surface,
Lemma~\ref{lem:dilation-zero-flat} gives $\Xi\ne0$ and shows that
some outer curvature field $b_{ij}$ is not identically zero. Lemma~\ref{lem:dilation-ambient-detection} and
\eqref{eq:dilation-boundary-flux} provide an ambient normal field
$\Psi=P^\perp V$ with $Q_\Sigma(\Xi,\Psi)>0$.
Lemma~\ref{lem:dilation-negative-extension}, applied within the
ambient subspace, gives index at least $n+1$, while the model ambient
index is $n$ by Corollary~\ref{cor:spectral-general-index}.
\end{proof}

\section{Curvature criteria and nonlinear local rigidity}
\label{sec:nonlinear-local-rigidity}
We prove local rigidity directly from the Dirichlet form, without
passing through a nonlinear parametrization theorem at the $T$-point.

\subsection{The curvature of a triple junction}
\begin{lemma}\label{lem:dilation-junction-curvature}
At a junction point, let $T$ be a unit tangent, let $\tau_j$ be the
outward unit conormals of the three incident faces, and put
$h_j=|\II_j|$ and $\kappa_j=\ip{D_TT}{\tau_j}$. Then
\begin{equation}\label{eq:dilation-K-bound}
 |D_TT|^2\leq\frac13\sum_{j=1}^3h_j^2.
\end{equation}
In particular, $|\kappa_j|\leq\max_l h_l$ at that point.
\end{lemma}
\begin{proof}
Set $K=D_TT$ and $E=\operatorname{span}\{\tau_1,\tau_2,\tau_3\}$.
The $120^\circ$ balance gives
$\sum_j\tau_j\otimes\tau_j=\tfrac32P_E$. By the Gauss formula,
$\II_j(T,T)=K-\kappa_j\tau_j$, and therefore
\[
 \sum_j\kappa_j^2=\frac32|P_EK|^2,
 \qquad
 \sum_j|\II_j(T,T)|^2
 =3|K|^2-\frac32|P_EK|^2\geq\frac32|K|^2.
\]
In the orthonormal frame $(T,\tau_j)$, minimality gives
$h_j^2=2|\II_j(T,T)|^2+2|\II_j(T,\tau_j)|^2$.
Summing proves \eqref{eq:dilation-K-bound}. Finally,
$|\kappa_j|\leq|K|\leq\max_l h_l$.
\end{proof}

\subsection{A geometric Dirichlet coercivity criterion}
For scalar functions on the abstract faces with zero outer trace,
choose finite positive constants $C_P,C_J$ such that
\begin{equation}\label{eq:dilation-scalar-constants}
 \begin{split}
 \sum_{i<j}\int_{S_{ij}}f_{ij}^2\,dA
 &\leq C_P\sum_{i<j}\int_{S_{ij}}|\nabla f_{ij}|^2\,dA,\\
 \sum_{i<j}\sum_{k\in\{i,j\}}\int_{I_k}f_{ij}^2\,ds
 &\leq C_J\sum_{i<j}\int_{S_{ij}}|\nabla f_{ij}|^2\,dA.
 \end{split}
\end{equation}
No compatibility between the scalar functions on different faces is
required. Since each outer arc is a nonempty regular boundary arc, the
first estimate is the standard Poincar\'e--Friedrichs inequality with
a nonempty Dirichlet portion of the boundary, and the trace theorem
together with the first estimate gives the second. Neither depends on
the face being a disk, and each incident junction edge is counted in
the second sum.

\begin{proposition}\label{prop:dilation-curvature-criterion}
Let
\[
 h=\max_{i<j}\|\II_{ij}\|_{L^\infty(S_{ij})}.
\]
If $C_Ph^2+C_Jh<1$, then $\Sigma$ is the standard flat $T$-cone,
up to an orthogonal transformation, with each face covered once.
More precisely, for every $\Phi\in\mathcal V_\Sigma^D$,
\begin{equation}\label{eq:dilation-coercivity}
 Q_\Sigma(\Phi,\Phi)
 \geq(1-C_Ph^2-C_Jh)
        \sum_{i<j}\int_{S_{ij}}|\nabla^\perp\Phi_{ij}|^2\,dA.
\end{equation}
\end{proposition}
\begin{proof}
Write
$G=\sum_{i<j}\int_{S_{ij}}|\nabla^\perp\Phi_{ij}|^2\,dA$.
The curvature endomorphism satisfies
$\ip{\mathcal A_{ij}\Phi_{ij}}{\Phi_{ij}}
 \leq|\II_{ij}|^2|\Phi_{ij}|^2$, while
Lemma~\ref{lem:dilation-junction-curvature} gives
$|\kappa_{ij,k}|\leq h$. Since $\Phi$ has zero outer trace, the
spherical boundary term in $Q_\Sigma$ vanishes, so
\[
 Q_\Sigma(\Phi,\Phi)
 \geq G-h^2\sum_{i<j}\int_{S_{ij}}|\Phi_{ij}|^2\,dA
 -h\sum_{i<j}\sum_{k\in\{i,j\}}\int_{I_k}|\Phi_{ij}|^2\,ds.
\]
Apply \eqref{eq:dilation-scalar-constants} to
$f_{ij}=|\Phi_{ij}|$. Kato's inequality
$|\nabla|\Phi_{ij}||\leq|\nabla^\perp\Phi_{ij}|$ almost everywhere
gives the two bounds by $C_PG$ and $C_JG$, respectively, and proves
\eqref{eq:dilation-coercivity}. The scalar Poincar\'e estimate also
gives $\|\Phi\|_{L^2}^2\leq C_PG$, so
\[
 \lambda_1^D(\Sigma)\geq\frac{1-C_Ph^2-C_Jh}{C_P}>0.
\]
Theorem~\ref{thm:Dirichlet-rigidity} gives the stated rigidity and
one-sheeted conclusion.
\end{proof}

\subsection{Uniform constants on the reference sectors}
\begin{lemma}\label{lem:dilation-uniform-constants}
Suppose each face has a parametrization by
$\overline{S_\vartheta}$ taking its three boundary arcs to the
corresponding arcs, and its pullback metric satisfies
\[
 m^2g_0\leq g_{ij}\leq M^2g_0,\qquad 0<m\leq M,
\]
where $g_0$ is the Euclidean sector metric. Then the inequalities in
\eqref{eq:dilation-scalar-constants} hold with
\begin{equation}\label{eq:dilation-explicit-constants}
 C_P=\frac{M^3}{4m},\qquad
 C_J=\frac{M^2}{m\sin(\vartheta/2)}
     =\frac{M^2}{m}\sqrt{\frac32}.
\end{equation}
These are explicit sufficient choices; no optimality is claimed.
\end{lemma}
\begin{proof}
For a scalar $f$ of zero outer trace on the flat sector, the scalar
argument used in \eqref{eq:dilation-flat-Poincare} gives
$\int f^2\leq\tfrac14\int|Df|^2$.
For the two ray traces, use the constant planar vector
\[
 q=(-\cot(\vartheta/2),-1).
\]
Its scalar product with the outward unit conormal on either ray
is one, and $|q|=1/\sin(\vartheta/2)$. The divergence theorem yields
\[
 \int_{\gamma_0\cup\gamma_\vartheta}f^2\,ds_0
 =2\int_{S_\vartheta}f\,q\cdot Df\,dA_0
 \leq\frac{2}{\sin(\vartheta/2)}\|f\|_{L^2}\|Df\|_{L^2}
 \leq\frac1{\sin(\vartheta/2)}\int_{S_\vartheta}|Df|^2\,dA_0.
\]
The outer circular contribution is zero because $f$ has zero trace
there. The divergence identity remains valid for $f\in H^1$, since
$f^2\in W^{1,1}$ and its boundary trace is the square of the $H^1$
trace, or equivalently by smooth approximation on the Lipschitz
sector; no value is prescribed at the vertex.

The area element is at most $M^2dA_0$, and an edge length element
is at most $Mds_0$. In a $g_0$-orthonormal frame, the two eigenvalues
of $g$ lie in $[m^2,M^2]$. It follows directly that the eigenvalues of
$\sqrt{\det g}\,g^{-1}$ lie in $[m/M,M/m]$, and hence
\[
 \int|Df|^2\,dA_0\leq\frac{M}{m}
                   \int|\nabla^g f|^2\,dA_g.
\]
Combining the three comparisons proves
\eqref{eq:dilation-explicit-constants}, first on one face and then
by summation. Finally, $\cos\vartheta=-1/3$ yields
$\sin^2(\vartheta/2)=2/3$.
\end{proof}

\begin{proof}[Proof of Theorem~\ref{thm:nonlinear-local-rigidity}]
Compose with $A^{-1}$ and relabel. Since $Dp_{ij}$ is an isometry and
$\|DY_{ij}-Dp_{ij}\|_{\mathrm{op}}<1/10$, the singular values of
$DY_{ij}$ lie between $m=9/10$ and $M=11/10$, so Lemma~\ref{lem:dilation-uniform-constants} applies. Since $p_{ij}$ is linear, the
$C^2_*$ bound also gives
$\left(\sum_{a,b}|\partial_{ab}Y_{ij}|^2\right)^{1/2}<1/10$.
In these coordinates
$\II_{ab}=(\partial_{ab}Y_{ij})^\perp$. Since the eigenvalues of
$g^{-1}$ are at most $m^{-2}$,
\[
 |\II_{ij}|_g\leq m^{-2}
       \left(\sum_{a,b}|\partial_{ab}Y_{ij}|^2\right)^{1/2}
 <\frac{10}{81}<\frac18.
\]
The constants in \eqref{eq:dilation-explicit-constants} obey
\[
 C_P=\frac{1331}{3600}<\frac38,
 \qquad C_J=\frac{121}{90}\sqrt{\frac32}<\frac74.
\]
These estimates give
\[
 C_Ph^2+C_Jh
 <\frac38\frac1{64}+\frac74\frac18
 =\frac{115}{512}<\frac14<1.
\]
Proposition~\ref{prop:dilation-curvature-criterion} proves the result.
The conclusion also puts the central vertex at zero, although it was
not fixed a priori. The constants above and the radius $1/10$ are
sufficient rather than optimal.
\end{proof}

\begin{remark}
The theorem concerns nearby free boundary minimal $T$-surfaces in the
stated $C^2$ topology, not only infinitesimal deformations; it gives
no $C^1$, varifold, or area-excess rigidity, and the compatible face
parametrizations need not be conformal.
\end{remark}

\enlargethispage{2\baselineskip}

\end{document}